\documentclass[titlepage,11pt]{article}
\usepackage[utf8]{inputenc}
\usepackage{hyperref} 
\usepackage{amsmath}
\usepackage{amsthm}
\usepackage{epsfig}
\usepackage{epstopdf}
\usepackage{amsfonts}
\usepackage{amssymb, color}
\usepackage{kotex}
\usepackage{thmtools}
\usepackage{thm-restate}
\usepackage{cleveref}

\newcommand{\D}{\Delta}

\newtheorem{THM}{Theorem}[section]
\newtheorem{LEM}[THM]{Lemma}

\newtheorem{COR}[THM]{Corollary}
\newtheorem{PROP}[THM]{Proposition}
\newtheorem{DEF}[THM]{Definition}
\newtheorem{OBS}[THM]{Observation}
\newtheorem{QUE}[THM]{Question}

\newtheorem*{GS}{The Gy\'arf\'as-Sumner conjecture}

\title{Unavoidable Subtournaments in Tournaments with Large Chromatic Number}
\author{Ilhee Kim\thanks{December \& Company, Seoul, Republic of Korea} \and Ringi Kim\thanks{
Ringi Kim was  supported by the National Research Foundation of Korea(NRF) grant funded by the Korea government(MSIT) (No. NRF-2018R1C1B6003786).
Department of Mathematical Sciences, KAIST, Daejeon, Republic of Korea.
}
}
\date{\today}

\begin{document}

\maketitle

\begin{abstract}


For a set $\mathcal{H}$ of tournaments, 
we say $\mathcal{H}$ is \emph{heroic} 
if  every tournament, 
not containing any member of $\mathcal{H}$ as a subtournament, 
has bounded chromatic number. 
In~\cite{hero}, Berger et al. explicitly characterized all heroic sets containing one tournament.
Motivated by this result, we study heroic sets containing two tournaments.
We give  a necessary condition for a set containing two tournaments to be heroic. 
We also construct infinitely many minimal heroic sets of size two.
\end{abstract}

\section{Introduction}\label{SEC:intro}
All graphs and digraphs in this paper are simple.
For a graph $G$, the \emph{chromatic number} of $G$, denoted by $\chi(G)$, is the minimum number of colors needed to color vertices of $G$ in such a way that there are no adjacent vertices with the same color. 
Since the chromatic number of $G$ is lower bounded by its clique number $\omega(G)$  (the maximum number of pairwise adjacent vertices of $G$), there has been great interest in a class of graphs whose chromatic number is bounded by some function of its clique number. 
If $\mathcal{C}$ is  a class of graphs closed under induced subgraphs, and there exists a function $f$ such that $\chi(G)\le f(\omega(G))$ for every $G\in \mathcal{C}$, then we say $\mathcal{C}$ is \emph{$\chi$-bounded by a $\chi$-bounding function $f$}.
A well-known example of a $\chi$-bounded class is the class of \emph{perfect} graphs. 
(A perfect graph is a graph with the property that, $\chi(H)=\omega(H)$ for every its induced subgraph $H$.)
 Clearly, the identity function is a $\chi$-bounding function for the class of perfect graphs.

There are many results and conjectures about $\chi$-bounded classes which are obtained by forbidding certain families of graphs.
A well-known example is the strong perfect graph theorem~\cite{perfect} which states that the set of graphs $G$, such that neither $G$ nor its complement contains an induced odd cycle of length at least five, is $\chi$-bounded by the identity function.
Recently, three conjectures of Gy\'arf\'as~\cite{gyarfas3}, regarding $\chi$-bounded classes of graphs forbidding some infinite sets of cycles, were proved in a series of papers by Chudnovsky, Scott, Seymour and Spirkl. 
For a graph $G$ and a set $\mathcal{G}$ of graphs, we say $G$ is \emph{$\mathcal{G}$-free} if $G$ contains no members of $\mathcal{G}$ as induced subgraphs.
The three conjectures, which are now theorems, state as follows:
If $\mathcal{G}$ is either one of the following classes, then the set of all $\mathcal{G}$-free graphs is $\chi$-bounded.

\begin{itemize}
\item the set of all odd holes of length at least five (Scott and Seymour~\cite{hole1});
\item the set of all holes of length at least $\ell$ for some $\ell$ (Chudnovsky, Scott and Seymour~\cite{hole3});
\item the set of all odd holes of length at least $\ell$ for some $\ell$ (Chudnovsky, Scott, Seymour and Spirkl~\cite{hole8}).
\end{itemize} 

Another conjecture due to Gy\'arf\'as~\cite{gyarfas}, independently proposed by Sumner~\cite{sumner}, deals with $\chi$-bounded classes obtained by forbidding finite family $\mathcal{F}$ of graphs. By the random construction of Erd\H{o}s~\cite{erdos}, we know that for each $c$ and $g$, there exists a graph $G$ with $\chi(G)\ge c$ and minimum cycle length at least $g$. This implies that for the class of $\mathcal{F}$-free graphs to be $\chi$-bounded, it is necessary that $\mathcal{F}$ contains a forest.  The Gy\'arf\'as-Sumner conjecture asserts that the necessary condition is also sufficient.
\begin{GS}
Let $K$ be a complete graph and $F$ a forest. 
Then, there exists $c$ such that every $\{K,F\}$-free graph has chromatic number at most $c$.
\end{GS}

This conjecture is known to be true for several classes of forests~\cite{gyarfas5, gyarfas2, gyarfas7, gyarfas6, hole12}, but is mostly wide open.

In this paper, we are interested in a similar question to the Gy\'arf\'as-Sumner conjecture for tournaments.
A \emph{tournament} is a digraph of which underlying graph is a complete graph. For a tournament $T$ and a set $S\subseteq V(T)$, we denote by $T|S$ the subtournament of $T$  induced on $S$. 
We say $S\subseteq V(T)$ is \emph{transitive} if $T|S$ has no directed cycle.

For a tournament $T$ and its vertices $u$ and $v$, if $uv \in E(T)$, then we say \emph{$u$ is adjacent to $v$} or \emph{$v$ is adjacent from $u$}. 
For two disjoint subsets $X$ and $Y$ of $V(T)$, if every vertex in $X$ is adjacent to every vertex in $Y$, then we say \emph{$X$ is complete to $Y$}, and write $X \Rightarrow Y$. 
If $T$ is obtained from the disjoint union of tournaments $T_1$ and $T_2$ by adding all edges from $V(T_1)$ to $V(T_2)$, we write $T=T_1\Rightarrow T_2$.

For tournaments $T_1$ and $T_2$, if $T_2$ is isomorphic to a subtournament of $T_1$, then we say \emph{$T_1$ contains $T_2$},  and if $T_1$ does not contain $T_2$, we say $T_1$ is \emph{$T_2$-free}. 
If $\mathcal{H}$ is a set  of tournaments and a tournament $T$ is $H$-free for every $H\in \mathcal{H}$, then we say $T$ is \emph{$\mathcal{H}$-free}.

For a positive integer $k$ and a tournament $T$, 
a \emph{$k$-coloring of $T$} is a map $\phi:V(G)\to C$ with $|C|=k$ such that $\phi^{-1}(c)$ is transitive for $c \in C$. 
The \emph{chromatic number} of a tournament $T$, denoted by $\chi(T)$, is the minimum $k$ such that $T$ admits a $k$-coloring. This tournament invariant was first introduced by Neumann Lara~\cite{tournamentcoloring}. 

In this paper, we study the tournament version of the  Gy\'arf\'as-Sumner conjecture, that is, we investigate 
a class $\mathcal{H}$ of tournaments where  every $\mathcal{H}$-free tournament has bounded chromatic number.
Such a set is called \emph{heroic}. (A heroic set for graphs can be defined similarly. We direct the interested reader to~\cite{graphhero}.)

\begin{DEF}
A set $\mathcal{H}$  of tournaments is \emph{heroic} if there exists $c$ such that every $\mathcal{H}$-free tournament has chromatic number at most $c$. 
\end{DEF}

For example, if $\mathcal{H}$ contains a cyclic triangle, then $\mathcal{H}$ is heroic since every tournament with chromatic number at least three contains a cyclic triangle.

\subsection{Tournaments with large chromatic number}\label{sec:construction}

Several graph classes with large chromatic number 
are known~\cite{My1955,Kv1989,Lo1968,NeRo1989,AlKo2016}. 
A complete graph is a trivial example, and
a Mycielski graph is a non-trivial example, which have clique number two but arbitrarily large chromatic number. 
In contrast to graphs, few such classes of tournaments have been developed. One is introduced in~\cite{hero}, as follows:
\\
\\
{\bf Construction of $D_n$.}
If $T$ is a tournament and $(X,Y,Z)$ is a partition of $V(T)$ such that $X \Rightarrow Y$, $Y \Rightarrow Z$ and $Z \Rightarrow X$, 
 we call $(X,Y,Z)$ a \emph{trisection} of $T$, 
 and if $T|X$, $T|Y$ and $T|Z$ are isomorphic to tournaments $A$, $B$ and $C$, respectively, we write $T=\D(A,B,C)$. We denote by $I$ a one-vertex tournament. 
We construct tournaments $D_n$ as follows: $D_1=I$, and for $n\ge 2$, $D_n =\Delta(I, D_{n-1},D_{n-1})$. See Figure~\ref{Dn_fig}. 
\vspace{0.5cm}

In section~\ref{SEC:class_unbounded}, we prove that the chromatic number of $D_n$ is equal to $n$.

\begin{figure} [ht!]
\centering
\includegraphics[scale=0.6]{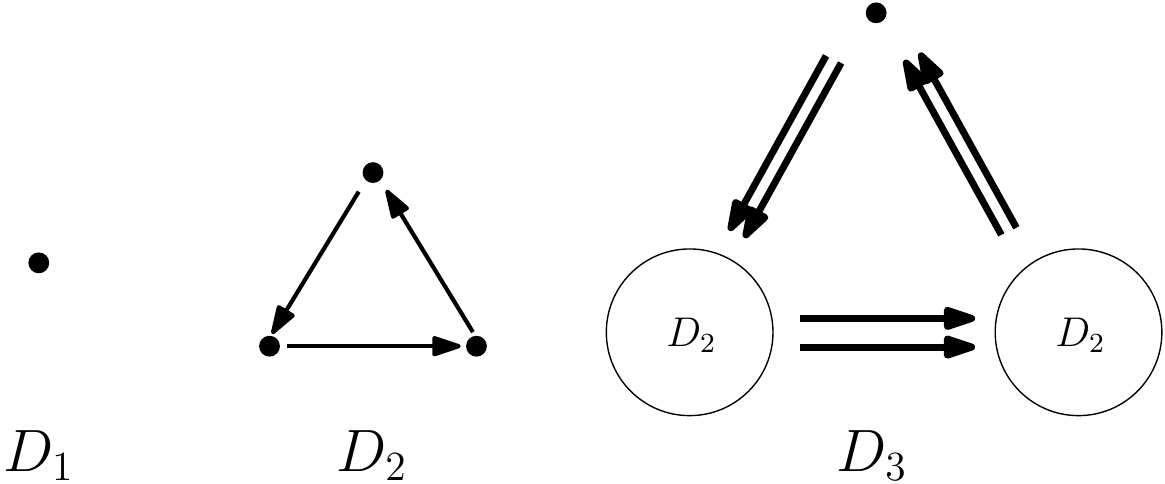}
\caption{$D_n$ for $n=1,2,3$}
\label{Dn_fig}
\end{figure}

In this section, 
we introduce another class of tournaments, which are denoted by $A_n$,
with large chromatic number. 
For a tournament $T$ and an integer $n\ge 2$, 
if  $(X_1,X_2,\ldots,X_{2n-1})$ is a partition of $V(T)$ such that 
 for $1\le i < j \le 2n-1$, 
\begin{itemize} 
\item  $V_j$ is complete to $V_i$ if both $i$ and $j$ are odd, and 
\item $V_i$ is complete to $V_j$ if either $i$ or $j$ is even,
\end{itemize}
then we call $(X_1,X_2,\ldots,X_{2n-1})$  a \emph{$\D$-partition} of $T$, and we write $T=\Delta(T_1,T_2,\ldots,T_{2n-1})$ where $T_i=T|V_i$ for $1\le i \le 2n-1$. Note that every trisection of a tournament is a $\Delta$-partition of it. 
\\
\\
{\bf Construction of $A_n$.}
$A_1$ is a one-vertex tournament, and for $n \ge 2$, $A_n=\D(I^{(1)},A_{n-1}^{(1)},I^{(2)},A_{n-1}^{(2)},\ldots,A_{n-1}^{(n-1)},I^{(n)})$  where each $I^{(i)}$ is isomorphic to $I$ (a one-vertex tournament) and $A_{n-1}^{(i)}$ is isomorphic to $A_{n-1}$. See Figure~\ref{An_fig}.
\vspace{0.5cm}

\begin{figure}[ht!]
\centering
\includegraphics[scale=0.8]{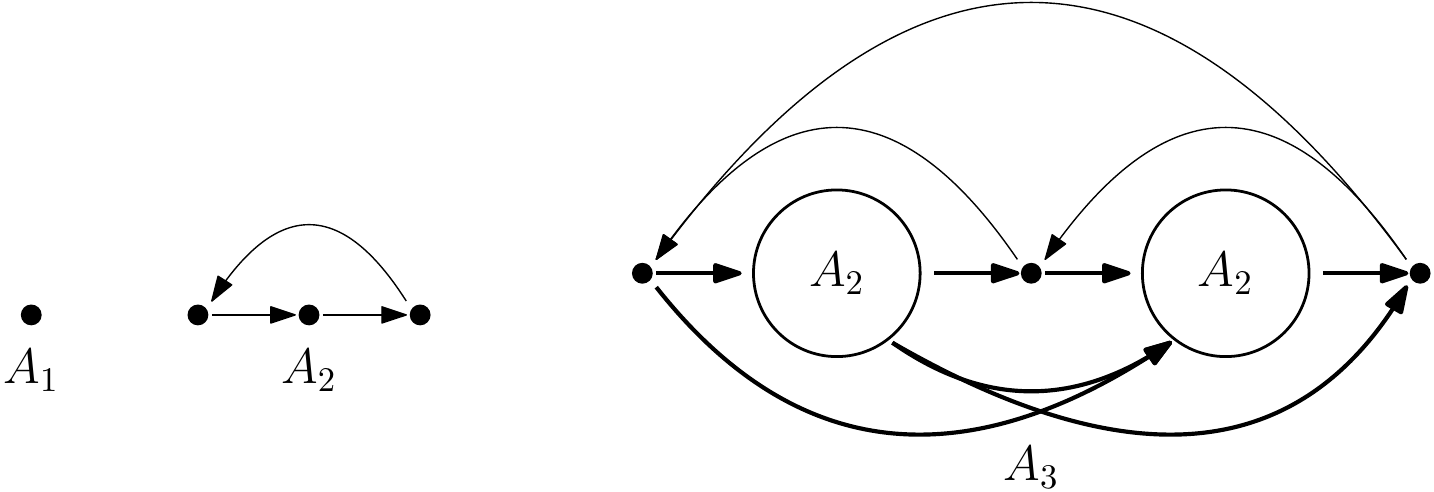}
\caption{$A_n$ for $n=1,2,3$. 
}
\label{An_fig}
\end{figure}

In section~\ref{SEC:class_unbounded}, we prove that $\chi(A_n)=n$.

\subsection{Tournaments contained in heroic sets}\label{SEC:intro-contained}
Let $\mathcal{T}$ be a set of tournaments. 
If, for every tournament $T\in \mathcal{T}$, every subtournament of $T$ is contained in $\mathcal{T}$, then $\mathcal{T}$ is said to be \emph{hereditary}.
If for every $n$, there exists a tournament in $\mathcal{T}$ with chromatic number larger than $n$, 
we say \emph{$\mathcal{T}$ has unbounded chromatic number}.
It is easy to see that if $\mathcal{T}$ is a hereditary class of tournaments and has unbounded chromatic number, then every heroic set contains a tournament in $\mathcal{T}$.

\begin{PROP}\label{PRO:hereditary}
Let $\mathcal{T}$ be a hereditary class of tournaments. If the chromatic number of $\mathcal{T}$ is unbounded, then every heroic set meets $\mathcal{T}$.
\end{PROP}

For a set $\mathcal{T}$ of tournaments, the \emph{closure of $\mathcal{T}$} is the minimal hereditary class of tournaments containing $\mathcal{T}$. We define two classes of tournaments as follows: 

\begin{itemize}
\item $\mathcal{D}$ is the closure of $\{D_n\mid n\ge 1\}$.
\item $\mathcal{A}$ is the closure of $\{A_n\mid n\ge1\}$.
\end{itemize} 

Since $\mathcal{D}$ and $\mathcal{A}$ have unbounded chromatic number, Proposition~\ref{PRO:hereditary} implies the following.
\begin{THM}\label{THM_meet1}
Every heroic set intersects with $\mathcal{D}$ and $\mathcal{A}$. 
\end{THM}
It is easy to see that two sets $\mathcal{D}$ and $\mathcal{A}$ are minimal in the
sense that there is no proper hereditary subset of $\mathcal{D}$ or $\mathcal{A}$ with unbounded chromatic number.

\subsection{Forest tournaments}
In the previous section, we constructed two (minimal) classes of tournaments intersecting with all heroic sets. In this section, we introduce another class of tournaments, which are called forest tournaments, intersecting with all finite heroic sets. 

If $S$ is a finite set and $\sigma$ is an ordering of $S$, then for $a,b \in S$, we write $a<_{\sigma} b$ if $a$ comes before $b$ in $\sigma$. 
For example, if $\sigma=s_1,s_2,\ldots,s_n$, then $s_i <_{\sigma} s_j$ for every $1\le i<j\le n$. 
If $S'$ is a subset of $S$, then $\sigma|S'$ is the sub-ordering of $S$ on $S'$.
We denote by $\sigma\setminus S'$ the ordering $\sigma|(S-S')$. 

For a tournament $T$ and an ordering $\sigma=v_1,v_2,\ldots,v_n$ of $V(T)$, an edge $v_iv_j$ of $T$ is called a {\em backward edge (under $\sigma$)} if $i>j$. The \emph{backedge graph $B_{\sigma}(T)$ of $T$ with respect to $\sigma$} is the ordered (undirected) graph with vertex set $V(T)$ and vertex ordering $\sigma$ such that $uv \in E(B_{\sigma}(T))$ if and only if either $uv$ or $vu$ is a backward edge of $T$ under $\sigma$.  See Figure~\ref{fig:backedge}.

\begin{figure} [ht!]
\centering
\includegraphics[scale=0.6]{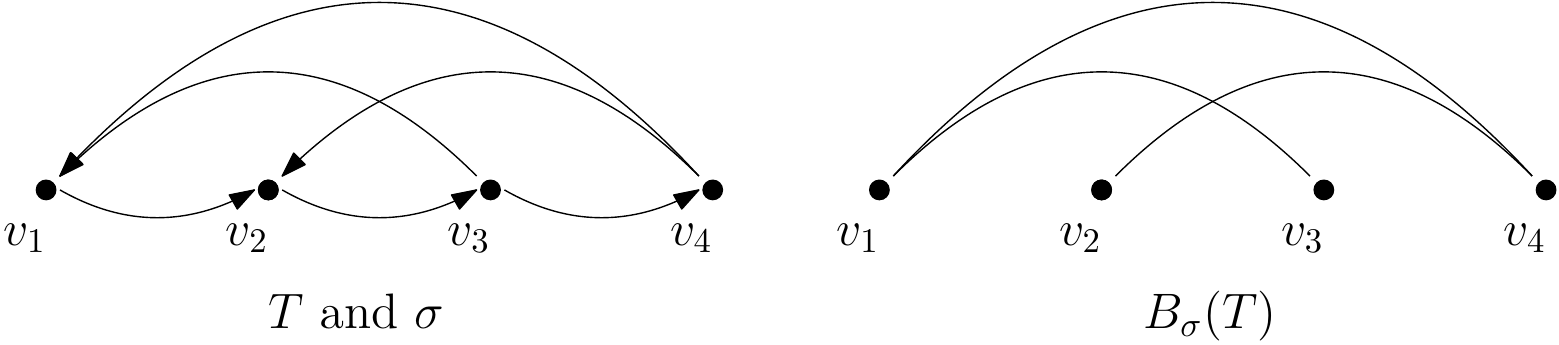}
\caption{$T$ and $B_{\sigma}(T)$}
\label{fig:backedge}
\end{figure}

The definition of a forest tournament is as follows:

\begin{DEF}

For a tournament $T$, a {\em forest ordering of $V(T)$} is an ordering $\sigma=v_1,v_2,\ldots,v_n$ of $V(T)$ such that 
\begin{itemize}
\item there exists $i$ such that no two edges of $B_{\sigma}(T)$ between $\{v_1,v_2,\ldots,v_i\}$ and $\{v_{i+1},\ldots,v_n\}$ are in the same component of $B_{\sigma}(T)$, 
and sub-orderings $v_1,v_2,\ldots,v_i$ and $v_{i+1},\ldots,v_n$ of $\sigma$ are forest orderings of $T|\{v_1,v_2,\ldots,v_i\}$ and $T|\{v_{i+1},\ldots,v_n\}$, respectively.
\end{itemize}
If such an ordering exists, we say $T$ is a {\em forest tournament} and the partition  $(\{v_1,v_2,\ldots,v_i\}, \{v_{i+1},\ldots,v_n\})$ of $V(T)$ is a {\em forest cut of $T$ (under $\sigma$)}.
\end{DEF}

For example, in Figure~\ref{forestexample}, $T$ is a forest tournament with  forest cut $(\{v_1,v_2,v_3\}, \{v_4,v_5,v_6,v_7\})$
since $T|\{v_1,v_2,v_3\}$ and $T|\{v_4,v_5,v_6,v_7\}$ are forest tournaments with forest cut $(\{v_1\},\{v_2,v_3\})$ and $(\{v_4,v_5\}, \{v_6,v_7\})$, respectively.

In section~\ref{SEC:forest}, we will show the following theorem.
\begin{THM}\label{THM_meet2}
Every finite heroic set contains a forest tournament.
\end{THM}

\begin{figure} [ht!]
\centering
\includegraphics[scale=0.6]{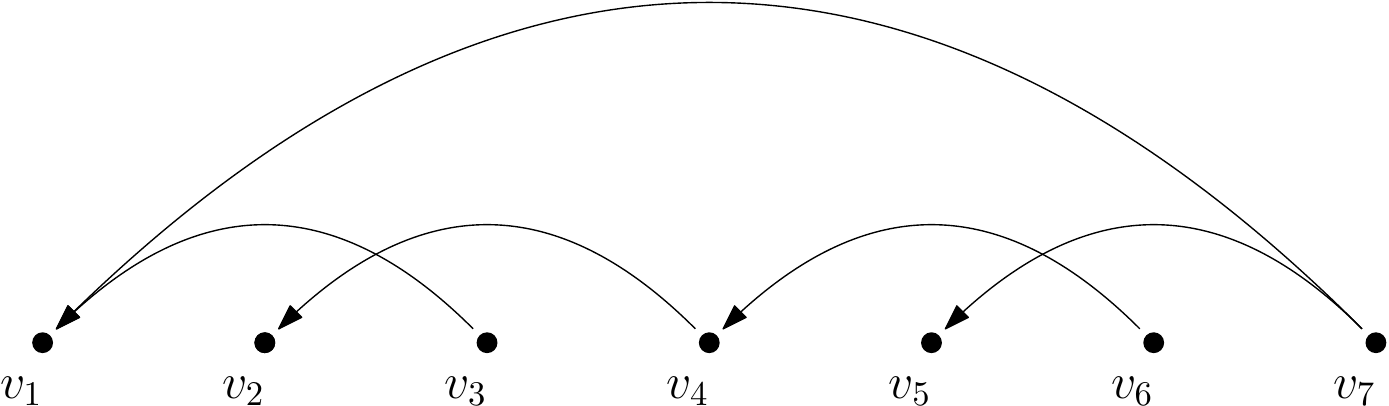}
\caption{A tournament $T$ with ordering $\sigma$ and its backward edges.}
\label{forestexample}
\end{figure}

We will also show some properties of forest tournaments. For example, we will show that the backedge graph $B_{\sigma}(T)$ of a forest tournament $T$ with forest ordering $\sigma$ is a forest, which is the reason that we call this tournament a forest tournament. 
We also prove that every forest tournament has chromatic number at most two, which shows the existence of (infinite) heroic sets not containing any forest tournaments. (e.g. the set of all tournaments with chromatic number three is heroic, but it does not contain any forest tournaments.)

\subsection{Small Heroic sets}

A tournament $H$ is called a \emph{hero} if every $H$-free tournament has bounded chromatic number, that is, $H$ is a hero if and only if $
\{H\}$ is heroic. In~\cite{hero}, Berger et al. explicitly characterized every \emph{hero} as follows:

\begin{THM}[Berger et al.~\cite{hero}]\label{hero}
Let $H$ be a tournament.
\begin{itemize}
\item[(1)] $H$ is a hero if and only if every strong component of $H$ is a hero.
\item[(2)] If $H$ is strongly connected, then $H$ is a hero if and only if $H=\D(I,H_1,H_2)$ where $H_1$ and $H_2$ are heroes and one of them is transitive.
\end{itemize}
\end{THM}

This result motivated us to the study of small heroic sets, in particular, heroic sets containing two tournaments.
For a set $\mathcal{H}$ consisting of two tournaments to be heroic,  
it must contain some tournament $D$ in $\mathcal{D}$ by Theorem~\ref{THM_meet1}.
If $D$ is a hero, then no matter what the other tournament in $\mathcal{H}$ is, $\mathcal{H}$ is heroic. Thus, the only interesting case is when $D$ is a non-hero.
Every non-hero in $\mathcal{D}$ is characterized as follows:

\begin{restatable}{LEM}{reD}\label{lem:nonhero-d3}
For a tournament $D\in \mathcal{D}$, $D$ is a non-hero if and only if $D$ contains $D_3$.
\end{restatable}
We will prove this lemma in section~\ref{SEC:nonhero}.

Let $\mathcal{D}'=\{D\in \mathcal{D} \mid \text{$D$ contains $D_3$}\}$, that is, the set of all non-heroes in $\mathcal{D}$.
For a tournament $F$, we say a tournament $H$ is an \emph{$F$-hero} if there exists $c$ such that every $\{F,H\}$-free tournament $T$ has chromatic number at most $c$.
For example, if $F$ is a hero, then every tournament is an $F$-hero.
By answering the following question, we can characterize all heroic sets consisting of two tournaments.

\begin{QUE}
Let $D\in \mathcal{D}'$. Which tournaments are $D$-heroes?
\end{QUE}

In this paper, we give a necessary condition for a tournament $H$ to be a $D$-hero for $D\in \mathcal{D}'$. Let $L_k$ be a transitive tournament with $k$ vertices.

\begin{THM}\label{THM:main2}
Let $D$ be a tournament in $\mathcal{D}'$. If a tournament $H$ is  a $D$-hero, then $H$ is isomorphic to one of the following. (See Figure~\ref{FIG:maintheorem}.)
\begin{itemize}
\item[1)]  $I$;
\item[2)] $H_1\Rightarrow H_2$ for some $D$-heroes $H_1$ and $H_2$;
\item[3)] $\Delta(I,L_k,H')$ or $\Delta(I,H',L_k)$ for some integer $k$ and some $D$-hero $H'$;
\item[4)] $\D(L_{k_1},I,L_{k_2},L_{k_3}, I)$ or $\D(I,L_{k_1},L_{k_2},I, L_{k_3})$ for some integers $k_1,k_2,k_3$;
\item[5)] $\D(I,H',L_{k_1},L_{k_2},I)$ or $\D(I,L_{k_1},L_{k_2},H',I)$ for some integers $k_1$, $k_2$ and some $D$-hero $H'$;
\item[6)] $\D(I,L_{k_1}, L_{k_2},I,L_{k_3},L_{k_4},I)$ for some integers $k_1,k_2,k_3,k_4$. 
\end{itemize}
\end{THM} 
We prove Theorem~\ref{THM:main2} in section~\ref{SEC:mainproof}.

\begin{figure} [ht!]
\centering
\includegraphics[scale=0.6]{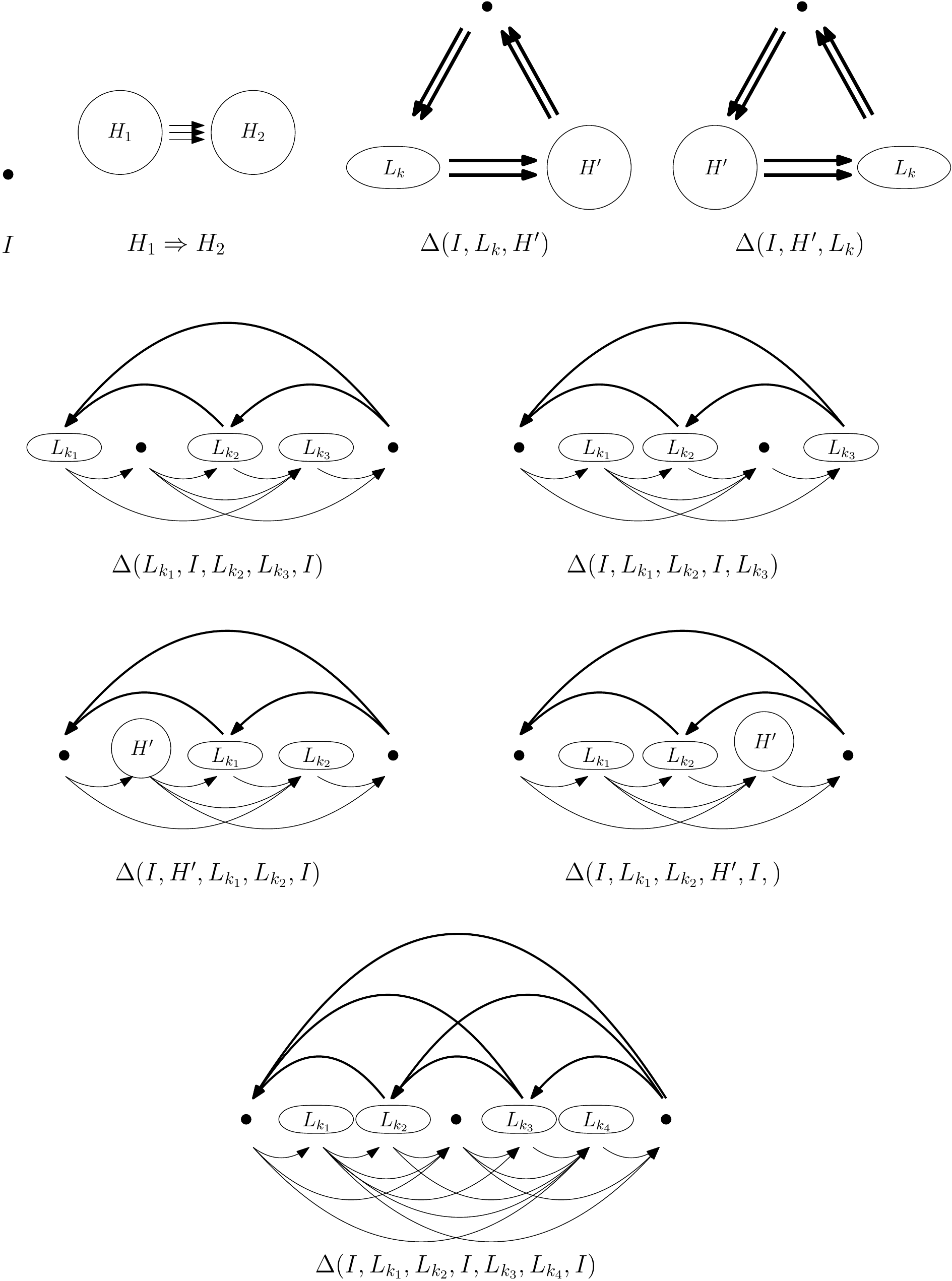}
\caption{Tournaments in Theorem~\ref{THM:main2}
}
\label{FIG:maintheorem}
\end{figure}

\subsection{$D$-heroes for $D\in \mathcal{D}'$}
The first result of Theorem~\ref{hero} also holds for $D$-heroes.

\begin{THM}\label{THM:growing1}
Let $D$ be a tournament in $\mathcal{D}'$. 
Then, a tournament $H$ is a $D$-hero if and only if every strong component of $H$ is a $D$-hero.
\end{THM}

Theorem~\ref{THM:growing1} is straightforward by the following lemma proved in~\cite{hero}.
\begin{LEM}[Berger et al. \cite{hero}]\label{lemma1}
Let $\mathcal{H}_1,\mathcal{H}_2$ be sets of tournaments such that every member of $\mathcal{H}_1\cup \mathcal{H}_2$ has at most $c\,(\ge 3)$ vertices. Let $\mathcal{H}=\{H_1\Rightarrow H_2 \mid H_1\in \mathcal{H}_1, H_2\in \mathcal{H}_2\}$. 
For every $\mathcal{H}$-free tournament $T$, if every $\mathcal{H}_1$-free subtournament of $T$ and $\mathcal{H}_2$-free subtournament of $T$ has chromatic number at most $c$, then  the chromatic number of $T$ is at most $(2c)^{4c^2}$.
\end{LEM}

We remark that Lemma~\ref{lemma1} also implies (1) of Theorem~\ref{hero}.

In contrast to (1) of Theorem~\ref{hero}, the second result does not hold for $D$-heroes in general. 
(Theorem~\ref{THM:U3} will give an example of a $D$-hero which is strongly connected but does not admit a trisection.)  
However, it turns out that 
(2) of Theorem~\ref{hero} holds for $D$-heroes admitting a trisection.

\begin{THM}\label{THM:growing2}
Let $D$ be a tournament in $\mathcal{D}'$. 
Let $H$ be a tournament admitting a trisection. 
Then, $H$ is a $D$-hero if and only if $H$ is either $\D(I,H',L_k)$ or $\D(I,L_k,H')$ where $k$ is a positive integer and $H'$ is a $D$-hero.
\end{THM}

Alghough the proof of Theorem~\ref{THM:growing2} is the same as that of (2) of Theorem~\ref{hero} in~\cite{hero},  we give the proof in section~\ref{SEC:growing} for reader's convenience.

The smallest tournament in the list of Theorem~\ref{THM:main2}, which cannot be obtained from Theorem~\ref{THM:growing1} or Theorem~\ref{THM:growing2}, is $\D(I,I, I,I,I)$. 
 We simply denote this tournament by $U_3$. 
In the following theorem, we show that $U_3$ is a $D$-hero for every $D\in \mathcal{D}'$.

\begin{THM}\label{THM:U3}
Let $D\in \mathcal{D}'$. Then, $U_3$ is a $D$-hero.
\end{THM}
The proof will be given in section~\ref{SEC:proofU3}.

Generalizing the definition of $U_3$, 
let $U_n=\D(I^{(1)},I^{(2)}, \ldots,I^{(2n-1)})$, that is, the tournament with $V(U_n)=\{v_1,v_2,\ldots,v_{2n-1}\}$ such that for $1\le i<j \le 2n-1$, $v_j$ is adjacent to $v_i$ if and only if both $i$ and $j$ are odd. See Figure~\ref{Un_fig}.

If $n\ge 5$, then $U_n$ is not contained in the list in Theorem~\ref{THM:main2}, and if $n\le 2$, then $U_n$ is either a one-vertex tournament or a cyclic triangle, which is a trivial $D$-hero. And by Theorem~\ref{THM:U3}, we know that $U_3$ is a $D$-hero for every $D\in \mathcal{D}'$. The only remaining case is that $n=4$.
So, we finish this section with the following question.
\begin{figure} [ht!]
\centering
\includegraphics[scale=0.55]{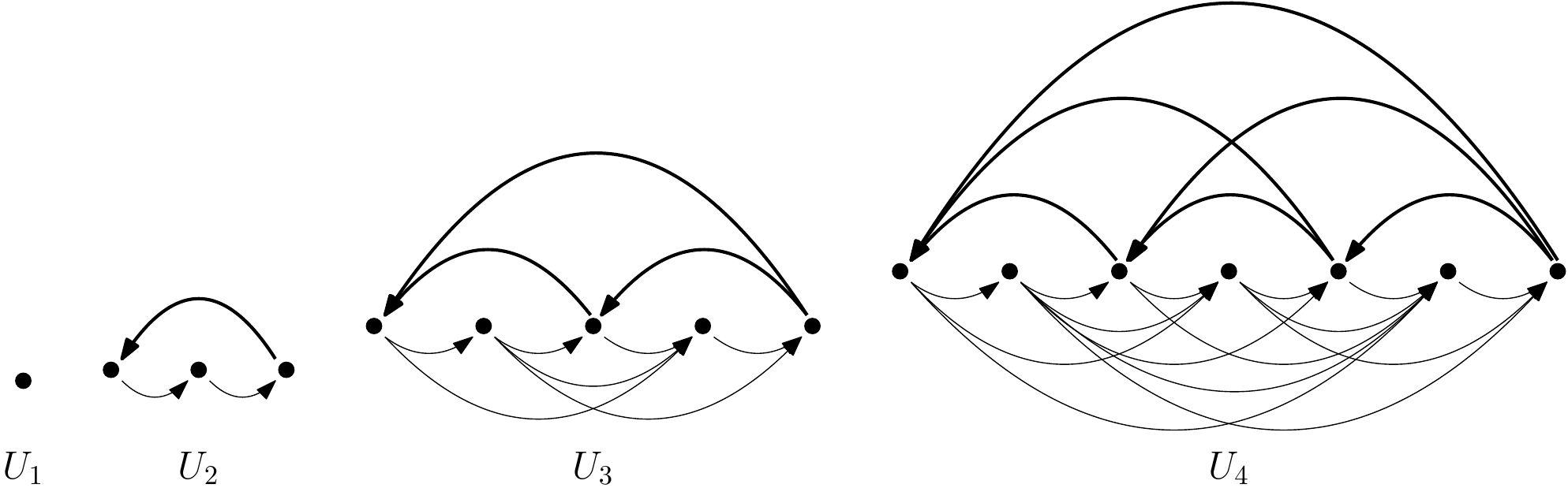}
\caption{$U_n$ for $n=1,2,3,4$. 
}
\label{Un_fig}
\end{figure}

\begin{QUE}\label{question}
For which tournaments $T \in \mathcal{D'}$, is $U_4$ a $T$-hero? In particular, is $U_4$ a $D_3$-hero?
\end{QUE}

\section{Classes of tournaments with unbounded chromatic number}\label{SEC:class_unbounded}
In this section, we prove that $\chi(D_n)=\chi(A_n)=n$, which directly implies Theorem~\ref{THM_meet1}.

\begin{PROP}\label{LEM:Dn}
For every positive integer $n$, $\chi(D_n)=n$.
\end{PROP}

\begin{proof}
We proceed by induction on $n$.
If $n=1$, then $|V(D_1)|=1$, so $\chi(D_1)=1$.

Let $n\ge 2$, and suppose $\chi(D_k)=k$ for all $k <n$. 
Let $(X_1,X_2,X_3)$ be a trisection of $D_n$ such that $|X_1|=1$ and  $X_2$ and $X_3$ induce $D_{n-1}$.  Let $X_1=\{x_1\}$ and $\phi_i:X_i \to [n-1]$ be an $(n-1)$-coloring of $D_n|X_i$ for $i=2,3$. Such colorings exist by the induction hypothesis. Let $\phi:V(D_n)\to [n]$ be a map such that $\phi(x_1)=n$ and for $v\in V_i$, $\phi(v)=\phi_i(v)$ for $i=1,2$.
Then, clearly, $\phi$ is an $n$-coloring of $D_n$, so $\chi(D_n) \le n$.

To show $\chi(D_n)\ge n$, suppose there exists an $(n-1)$-coloring $\psi: V(D_n) \to [n-1]$ of $D_n$. Since $\chi(D_{n-1})=n-1$, it follows that $|\psi(X_2)|=|\psi(X_3)|=n-1$. We may assume that $\psi(x_1)=n-1$. Let $x_2\in X_2$ and $x_3 \in X_3$ be vertices with $\psi(x_2)=\psi(x_3)=n-1$. Such $x_2$ and $x_3$ exist as $|\psi(X_2)|=|\psi(X_3)|=n-1$. 
Then, $\{x_1,x_2,x_3\}$ induces a monochromatic cyclic triangle in $D_n$ which yields a contradiction. Therefore $D_{n}$ is not $(n-1)$-colorable, implying that $\chi(D_n)= n$. This completes the proof.
\end{proof}

\begin{PROP}\label{LEM:An}
For every positive integer $n$, $\chi(A_n)=n$.
\end{PROP}

\begin{proof}
We proceed by induction on $n$.
If $n=1$, then $|V(A_1)|=1$, so $\chi(A_1)=1$.

Let $n\ge 2$, and assume the chromatic number of $A_k$ is equal to $k$ for all $k<n$.
Let $(\{v_1\},X_1,\{v_2\},X_2,\ldots,X_{n-1},\{v_n\})$ be a $\D$-partition of $A_n$ where $A_n|X_j$ is isomorphic to $A_{n-1}$ for $1\le j \le n-1$.
Let $\phi_j: X_j \to [n-1]$ be an $(n-1)$ coloring of $A_n|X_{j}$ for $j=1,2,\ldots,n-1$. 
Then, the map $\phi:V(A_n) \to [n]$ defined as, $\phi(v_i)=n$ for $i=1,2,\ldots,n$ and $\phi(v)=\phi_j(v)$ if  $v\in X_j$ for $j=1,2,\ldots,n-1$, is an $n$-coloring of $A_n$.

To prove that $A_n$ is not $(n-1)$-colorable, let us assume that there exists an $(n-1)$-coloring $\psi:V(A_n) \to [n-1]$ of $A_n$. 
Since $\psi$ is an $(n-1)$-coloring, there exist two vertices $v_p,v_q$ with $\psi(v_p)=\psi(v_q)$ and $p<q$ by the pigeonhole principle. 
We may assume that $\psi(v_p)=\psi(v_q)=n-1$.
Since $A_{n}|X_p$ is not $(n-2)$-colorable, it follows that $\psi(X_p)=[k-1]$, and there exists $y \in X_p$ such that $\psi(y)=n-1$. 
Then, $\{y,v_p,v_q\}$ induces a monochromatic cyclic triangle, a contradiction.
This completes the proof.
\end{proof}

In the remaining of this section, we investigate properties of tournaments in $\mathcal{A}$.

\begin{PROP}\label{PROP:upartition}
If a tournament $T\in \mathcal{A}$ is strongly connected, then there exists a $\D$-partition of $T$.
\end{PROP}

\begin{proof}
Take the minimal $n$ such that $A_n$ contains $T$. 
We consider a $\D$-partition  $(\{v_1\},X_1,\{v_2\},X_2,\ldots,X_{n-1},\{v_n\})$ of $A_n$ where $A_n|X_j$ is isomorphic to $A_{n-1}$ for $1\le j \le n-1$.

Let $B=V(T)\cap \{v_1,v_2,\ldots,v_n\}$. 
If $B$ is empty, 
then let $m$ be the minimum such that $V(T)\cap X_m \neq \emptyset$. 
Since $V(T) \not \subseteq X_m$ by the minimality of $n$, it follows that $V(T) \setminus X_m$ is not empty.
So, $V(T)\cap X_m$ is complete to $V(T)\setminus X_m$ in $T$, which yields a contradiction that $T$ is strongly connected. Therefore, $B\neq \emptyset$. 

Let $B=\{v_{i_1},v_{i_2},\ldots,v_{i_k}\}$ with $i_1<i_2<\cdots <i_k$. 
Observe that $V(T)\setminus B \subseteq \bigcup_{j=i_1}^{i_k-1}X_{j}$ since  $T$ is strongly connected. 
For $1\le j \le k-1$, 
let $Y_j=V(T)\cap \left( \bigcup_{s=i_j}^{i_{j+1}-1}X_s\right)$. 
Then, $v_{i_j}$ is complete to $Y_{j'}$ for $j\le j'\le k-1$
and complete from $Y_{j''}$ for $1\le j'' \le j-1$.
So, $(\{v_{i_1}\},Y_1 ,\{v_{i_2}\},Y_2,\ldots,Y_{k-1},\{v_{i_k}\})$ is a $\D$-partition of $T$. This completes the proof.
\end{proof} 

For a tournament $T$, a subset $S \subseteq V(T)$ and a vertex $v$ outside of $S$, we say \emph{$v$ is mixed on $S$}, if $v$ has both an out-neighbor and an in-neighbor in $S$. A subset $S$ of $V(T)$ with $1<|S|<|V(T)|$ is called \emph{homogeneous} if every vertex outside of $S$ is not mixed on $S$. 

\begin{PROP}\label{Uhomo}
Let $T$ be a strong tournament in $\mathcal{A}$. If $S$ is a maximal homogeneous set of $T$ and  $(\{v_1\},X_1,\{v_2\},X_2,\ldots,X_{n-1},\{v_n\})$ is a  $\D$-partition of $T$, then $S=X_k$ for some $1\le k \le n-1$.
\end{PROP}
\begin{proof}
Let $B=\{v_1,v_2,\ldots,v_{n}\}$. Clearly, $S\not\subseteq B$.
Choose the smallest $m$ such that $X_m \cap S \neq \emptyset$, and let $x \in X_m \cap S$.

We claim $S\cap B= \emptyset$.
Suppose $S \cap B\neq \emptyset$, and let  $y \in S\cap B$. Since $v_1$ and $v_n$ are mixed on $\{x,y\}$, they belong to $S$. 
By the definition of a homogeneous set, there exists $z\in V(T)\setminus S$, but $z$ is mixed on $\{v_1,v_n\}$, a contradiction.
Therefore $S \cap B=\emptyset$. 

If $S \not \subseteq X_m$, then $T|S$ is not strongly connected since $S\cap X_m$ is complete to $S\setminus X_m$. So, it follows that $S\subseteq X_m$. 
Lastly, since $X_m$ is homogeneous and $S$ is maximal, $S=X_m$.
This completes the proof.
\end{proof}  

A tournament is  \emph{prime} if it does not have homogeneous sets. Observe that if a tournament $T$ has at least three vertices and is prime, then $T$ is strongly connected. 
Recall that  $U_n=\Delta(I^{(1)},I^{(2)},\ldots,I^{(2n+1)})$ where $I^{(i)}$ is a one-vertex tournament. It is easy to see that $U_n$ is prime.

\begin{PROP}\label{Uprime}
Let $T \in \mathcal{A}$ be a tournament with at least three vertices.  Then, $T$ is prime if and only if $T$ is isomorphic to $U_{n}$ for some integer $n\ge2$.
\end{PROP}
\begin{proof}
The `if' part is clear.

For the `only if' part, 
if $T$ is prime, then $T$ is strongly connected, and by Proposition~\ref{PROP:upartition}, there exists a $\D$-partition $(\{v_1\},X_1,\{v_2\},\ldots,X_{n-1},\{v_{n}\})$ of $T$.

If  $|X_i| \ge2$ for some $1\le i\le n$, then  $X_i$ is homogeneous, so, $|X_i| = 1$ for every $i$.  Therefore, $T$ is isomorphic to $U_n$.
\end{proof}

\section{Proof of Theorem~\ref{THM_meet2}}\label{SEC:forest}

Let $\mathcal{F}$ be the set of all forest tournaments.
First, we show that $\mathcal{F}$ is hereditary.

\begin{PROP}\label{closed}
Let $T$ be a forest tournament with at least two vertices and forest ordering $\sigma$. Then, for every $v\in V(T)$,  $T\setminus v$ is a forest tournament and $\sigma\setminus v$ is its forest ordering.
\end{PROP}

\begin{proof}
We use induction on the number of vertices of $T$. Let $|V(T)|=n$ and let $\sigma'=\sigma\setminus v$ and $T'=T\setminus v$.
 
If $n=2$, we are done since $B_{\sigma'}(T')$ has no edge.

Let $n>2$ and assume that Proposition~\ref{closed} is true for every forest tournament with less than $n$ vertices. 
Let $(V_1,V_2)$ be a forest cut of $T$ under $\sigma$, so $T|V_i$ is a forest tournament with forest ordering $\sigma|V_i$ for $i=1,2$. 
Without loss of generality, let $v \in V_1$. Let $\sigma_1=\sigma|V_1$. 
If $V_1=\{v\}$, then $T'$ is $T|V_2$ which is a forest tournament with  forest ordering $\sigma'=\sigma|V_2$.  
Thus we may assume that $|V_1|>1$.  
Then, $T|V_1$ is a forest tournament with forest ordering $\sigma_1$, and   by the induction hypothesis, $T|(V_1\setminus v)$ is a forest tournament with forest ordering $\sigma_1\setminus v$.
Therefore, $(V_1\setminus v,V_2)$ is a forest cut of $T'$ under $\sigma'$, and so $T'$ is a forest tournament with forest ordering $\sigma'$. This completes the proof.
\end{proof}

Next, we prove that for a forest tournament $T$ and its forest ordering $\sigma$, $B_{\sigma}(T)$ does not contain a cycle as an induced subgraph. 
For an ordered graph $G$ with at least two vertices and vertex ordering $\sigma=v_1,\ldots,v_n$, the {\em thickness of $G$ (under $\sigma$)} is the minimum number of edges between $\{v_1,v_2,\ldots,v_i\}$ and $\{v_{i+1},\ldots,v_n\}$ over all $i$'s. 

\begin{PROP}\label{onecomponent}
Let $T$ be a forest tournament with forest ordering $\sigma$. If $B_{\sigma}(T)$ is connected, then the thickness of $B_{\sigma}(T)$ is one.
\end{PROP}
\begin{proof}
Since $B_{\sigma}(T)$ is connected, its thickness is at least one.
Let $(V_1,V_2)$ be a forest cut of $T$ under $\sigma$. Since $B_{\sigma}(T)$ has one component, there is exactly one edge between $V_1$ and $V_2$ in $B_{\sigma}(T)$. So, the thickness of $B_{\sigma}(T)$ is one.
\end{proof}

\begin{COR}\label{nocycle}
Let $T$ be a forest tournament and $\sigma$ its forest ordering. Then, $B_{\sigma}(T)$ does not contain a cycle as an induced subgraph.
\end{COR}
\begin{proof}
Suppose there exists $V' \subseteq V(T)$ such that $B_{\sigma}(T)|V'$ is a cycle. Let $T'=T|V'$ and $\sigma' =\sigma|V'$. 
Then, $T'$ is a forest tournament and $\sigma'$ is its forest ordering by Proposition~\ref{closed}.
Since a cycle is connected, Proposition~\ref{onecomponent} implies that the thickness of $B_{\sigma'}(T')$ is one. However, for every partition $(V_1,V_2)$ of $V(B_{\sigma'}(T'))$, there exist at least two edges between $V_1,V_2$ since  $B_{\sigma'}(T')$ is a cycle, a contradiction. This completes the proof.
\end{proof}

A forest (undirected graph) has chromatic number at most two. It holds for a forest tournament as well. 

\begin{PROP}\label{twocolor}
Every forest tournament has chromatic number at most two.
\end{PROP}
\begin{proof}
Let $T$ be a forest tournament with forest ordering $\sigma$. 
By Corollary~\ref{nocycle}, $B_{\sigma}(T)$ is a forest, 
in particular, it is 2-colorable (as a graph coloring). 
So, $V(T)$ can be partitioned into two sets $(X,Y)$ such that no pair of adjacent vertices in the same set. 
Then, there is no backward edge in $T|X$ (resp. $T|Y$) under $\sigma|X$ (resp. $\sigma|Y$), which implies that $X$ and $Y$ are transitive sets in $T$. So, $\chi(T)\le 2$. 
\end{proof}

The remaining of this section is devoted to proving Theorem~\ref{THM_meet2}. 
We start with some definitions. 

For a tournament $T$  and an injective map $\phi :V(T)\to \mathbb{Z}^+$, let $\sigma_{\phi}$ be the ordering $v_1,v_2,\ldots,v_n$ of $V(T)$ such that $\phi(v_i) < \phi(v_j)$ for every $i<j$. For  a backward edge $e$ of $T$ under $\sigma_{\phi}$,  if its end vertices are $x$ and $y$, we define $\phi(e)=|\phi(x)-\phi(y)|$. For integers $r,s \ge 1$ and distinct $e,f \in E(B_{\sigma_{\phi}}(T))$ we say $e$ and $f$ are {\em $(r,s)$-comparable (under $\phi$)} if 
\begin{itemize}
\item there exists a path in $B_{\sigma_{\phi}}(T)$ with at most $s$ edges containing $e$ and $f$, and
\item $ \frac{1}{r} \le \frac{\phi(e)}{\phi(f)} \le r$.
\end{itemize}

For positive integers $r$ and $s$, we denote by $\mathcal{C}_{(r,s)}$ the class of tournaments $T$ such that there exists an injective map $\phi$ from $V(T)$ to $\mathbb{Z}$ such that no two edges of $B_{\sigma_{\phi}} (T)$ are $(r,s)$-comparable.
It is easy to see that $\mathcal{C}_{(r,s)}$ is hereditary.
The following is proved in~\cite{hero}.

\begin{LEM}[Berger et al.~\cite{hero}]\label{crs}
For integers $r,s \ge 1$, the chromatic number of $\mathcal{C}_{(r,s)}$ is unbounded. So, every heroic set meets $\mathcal{C}_{(r,s)}$. 
\end{LEM}

Lemma~\ref{crs} provides infinitely many hereditary classes $\mathcal{C}(r,s)$ of tournaments meeting every heroic set.
This implies that if $\mathcal{H}$ is a finite heroic set, then $\mathcal{H}$ contains a tournament $H$ belonging to $\mathcal{C}(r,s)$ for infinitely many pairs $(r,s)$. 
Observe that for $e,f \in E(T)$, 
if $e$ and $f$ are not $(r,s)$-comparable under $\phi$, 
then they are not $(r',s')$-comparable under $\phi$ for every positive integers $r' (\le r)$ and $s' (\le s)$. 
So, it follows that $\mathcal{C}_{(r,s)} \subseteq \mathcal{C}_{(r',s')}$, and 
it directly leads to the following lemma.
Let  $\mathcal{C}=\bigcap_{r,s \in \mathbb{Z}^+} \mathcal{C}_{(r,s)}$.

\begin{LEM}\label{comparable}
If $\mathcal{H}$ is a finite heroic set, then it contains $H$ such that $H \in \mathcal{C}_{(r,s)}$ for every positive integers $r$ and $s$. That is, 
$\mathcal{H}\cap \mathcal{C} \neq \emptyset$.
\end{LEM}

For a tournament $T$ and a positive integer $r$, we say an injective map $\phi:V(T)\to \mathbb{Z}^+$ is \emph{$r$-incomparable}, if
for every pair $(e,f)$ of edges of $B_{\sigma_{\phi}}(T)$ in the same component, $\frac{\phi(e)}{\phi(f)}$ is either greater than $r$ or less than $\frac{1}{r}$.
We note that for $r\ge r'$, if $\phi$ is $r$-incomparable, then it is $r'$-incomparable.
We say a vertex ordering $\sigma$ of $T$ is \emph{incomparable} if for every positive integer $r$, there exists an $r$-incomparable injective  map $\phi:V(T)\to \mathbb{Z}^+$  such that $\sigma=\sigma_{\phi}$.

\begin{LEM}\label{stableordering}
Let $T$ be a tournament.
Then, $T$ belongs to $\mathcal{C}$ if and only if there exists an incomparable vertex ordering of $T$.
\end{LEM}

\begin{proof}
The `if' part is clear by the definitions of $\mathcal{C}$ and an incomparable vertex ordering.

For the `only if' part, let $|V(T)|=n$.
For each integer $r\ge 1$, let $\phi_r$ be an injective map from $V(T)$ to $\mathbb{Z}^+$ with the property that 
no two edges of $B_{\sigma_{\phi_r}} (T)$ are $(r,n-1)$-comparable. Such $\phi_r$ exists by the definition of $\mathcal{C}$.
Since for every pair $(e,f)$ of edges in the same component of $B_{\sigma_{\phi_r}}(T)$, there exists a path $P$ with at most $n-1$ edges, with $e,f \in E(P)$,
it follows that  $\frac{\phi_r(e)}{\phi_r(f)}$ is either greater than $r$ or less than $\frac{1}{r}$.
So, $\phi_r$ is $r$-incomparable.

Since there are finitely many orderings of $V(T)$, there exists an ordering $\sigma$ of $V(T)$ which is equal to $\sigma_{\phi_r}$ for infinitely many positive integers $r$. 

We claim $\sigma$ is incomparable.
For every integer $r'\ge 1$, there exists $r\ge r'$ such that $\sigma=\sigma_{\phi_r}$. 
Since $\phi_r$ is $r'$-incomparable, $\sigma$ is incomparable.
\end{proof}

We are ready to prove Theorem~\ref{THM_meet2}. 

\begin{proof}[Proof of Theorem~\ref{THM_meet2}]
We will show that for a tournament $T$ and a vertex ordering $\sigma=v_1,v_2,\ldots,v_n$ of $T$, $\sigma$ is incomparable if and only if $\sigma$ is a forest ordering. Then, by Lemma~\ref{comparable}, Theorem~\ref{THM_meet2} is straightforward.

We use induction on $n$. It is clear when  $n=1$. Let $n>1$ and assume that the statement holds for every tournament with less than $n$ vertices.

Let $\sigma$ be incomparable. Let $\phi: V(T)\to \mathbb{Z}^+$ be a map such that
\begin{itemize}
\item[(1)] $\sigma=\sigma_{\phi}$, and
\item[(2)] for every $e,f \in B_{\sigma_{\phi}}(T)$ in the same component, $\frac{\phi(e)}{\phi(f)}$ is either greater than $n-1$ or less than $\frac{1}{n-1}$.
\end{itemize}
Take $v_i$ such that  $|\phi(v_i)-\phi(v_{i+1})|$ is maximized. Let $V_1=\{v_1,\ldots,v_i\}$ and $V_2=\{v_{i+1},\ldots,v_n\}$.

We claim that $(V_1,V_2)$ is a forest cut of $T$ under $\sigma$.
Suppose there exist two edges $e$ and $f$ between $V_1$ and $V_2$  contained in the same component of $B_{\sigma}(T)$.  
Without loss of generality, we may assume that $\frac{\phi(e)}{\phi(f)}>n-1$, so $\phi(e)>(n-1)\phi(f)$.
Then, it follows by (2) that 
$$
\phi(v_n)-\phi(v_1) \ge \phi(e) > (n-1) \phi(f) 
\ge (n-1) \left(\phi(v_{i+1})-\phi(v_i)\right).
$$
Since $\phi(v_{i+1})-\phi(v_i) \ge \phi(v_{j+1})-\phi(v_j)$ for every $j=1,2,\ldots,n-1$, it follows that $(n-1) \left(\phi(v_{i+1})-\phi(v_i)\right)
\ge \phi(v_n)-\phi(v_1)$, which yields a contradiction.
  Therefore, no two edges between $V_1$ and $V_2$ are in the same component of $B_{\sigma}(T)$.
  Moreover, for $i=1,2$, $\sigma|V_i$ is an incomparable vertex ordering of  $T|V_i$, so $\sigma|V_i$ is a forest ordering of $T|V_i$ by the induction hypothesis.  Hence, $\sigma$ is a forest ordering of $T$ with forest cut $(V_1,V_2)$.

Conversely, suppose $\sigma$ is a forest ordering of $T$ with a forest cut $(V_1,V_2)$ where $V_1=\{v_1,v_2,\ldots,v_i\}$ and $V_2=\{v_{i+1},\ldots,v_n\}$.
Let $T_1=T|V_1$ and $T_2=T|V_2$. By the induction hypothesis, $\sigma|V_i$ is an incomparable vertex ordering of $T|V_i$ for $i=1,2$. 

It is enough to show that for each $r \ge 1$, there exists an injective map $\phi:V(T)\to \mathbb{Z}^+$ satisfying the conditions (1) and (2) above.

For $i=1,2$, let $\phi_i$ be an $r$-incomparable map from $V_i$ to $\mathbb{Z}^+$  with $\sigma_{\phi_i}=\sigma|V_i$.
Let $\phi_1(v_i)-\phi_1(v_1)=a$ and $\phi_2(v_n)-\phi_2(v_{i+1})=b$. We define an injective map $\phi:V(T)\to \mathbb{Z}^+$ as follows:
$\phi(v_j)=\phi_1(v_j)$ for $1\le j \le i$ and $\phi(v_j)=\phi_1(v_i)+ab(r+1)^2+a(r+1)\phi_2(v_j)$ for $i+1\le j \le n$. 
Then, obviously,  $\sigma=\sigma_{\phi}$.

We claim that $\phi$ is $r$-incomparable.
Let $e$ and $f$ be edges of $B_{\sigma}(T)$ in the same component.
If either $e,f \in E(B_{\sigma_{\phi}|V_1}(T_1))$ or $e,f \in E(B_{\sigma_{\phi}|V_2}(T_2))$, we are done. 
If one is an edge of  $B_{\sigma_{\phi}|V_1}(T_1)$ and the other is an edge of $B_{\sigma_{\phi}|V_2}(T_2)$, say $e \in E(B_{\sigma_{\phi}|V_1}(T_1))$ and $f\in E(B_{\sigma_{\phi}|V_2}(T_2))$, then $\frac{\phi(f)}{\phi(e)} \ge r+1>r$ since $\phi(e)\le \phi(v_i)-\phi(v_1)=a$ and  $\phi(f) \ge a(r+1)$. 
So, we may assume that either $e$ or $f$ is an edge between $V_1$ and $V_2$, say $f$. Then, $e$ is contained in $B_{\sigma_{\phi}|V_1}(T_1)$ or $B_{\sigma_{\phi}|V_2}(T_2)$ since $e$ and $f$ are contained in the same component of $B_{\sigma_{\phi}}(T)$. 
Note that $\phi(f) \ge \phi(v_{i+1})-\phi(v_i) = ab(r+1)^2+a(r+1)\phi_2(v_{i+1})>ab(r+1)^2$, and $\phi(e)\le \max\{\phi_1(v_i)-\phi_1(v_1),a(r+1)\left(\phi_2(v_n)-\phi_2(v_{i+1})\right)\} \le ab(r+1)$. 
So, $\frac{\phi(f)}{\phi(e)} \ge r+1>r$. Therefore, $\phi$ is  $r$-incomparable.  
This completes the proof.
\end{proof}


\section{Minimal non-heroes}\label{SEC:nonhero}
Since every subtournament of a hero is a hero, it is interesting to characterize minimal non-heroes. In~\cite{hero}, the authors showed that there are only five minimal non-heroes.
Let $N$ be the tournament with five vertices $\{v_1,v_2,v_3,v_4,v_5\}$ such that $v_i$ is adjacent to $v_j$ for $2\le i<j\le 5$ and $v_1$ is complete to $\{v_2,v_4\}$ and complete from $\{v_3,v_5\}$.
Let $S_n$ be the tournament with $(2n-1)$ vertices $v_1,v_2,\ldots,v_{2n-1}$ such that $v_i$ is adjacent to $v_j$ if and only if $j-i \in \{1,2,\ldots,n-1\} \mod (2n-1)$.
Let $\D_2=\D(L_2,L_2,L_2)$.

\begin{THM}[Berger et al. \cite{hero}]
A tournament $H$ is a hero if and only if $H$ does not contain $D_3$, $U_3$, $N$, $S_3$ or $\D_2$.
\end{THM}

\begin{figure} [ht!]
\centering
\includegraphics[scale=0.5]{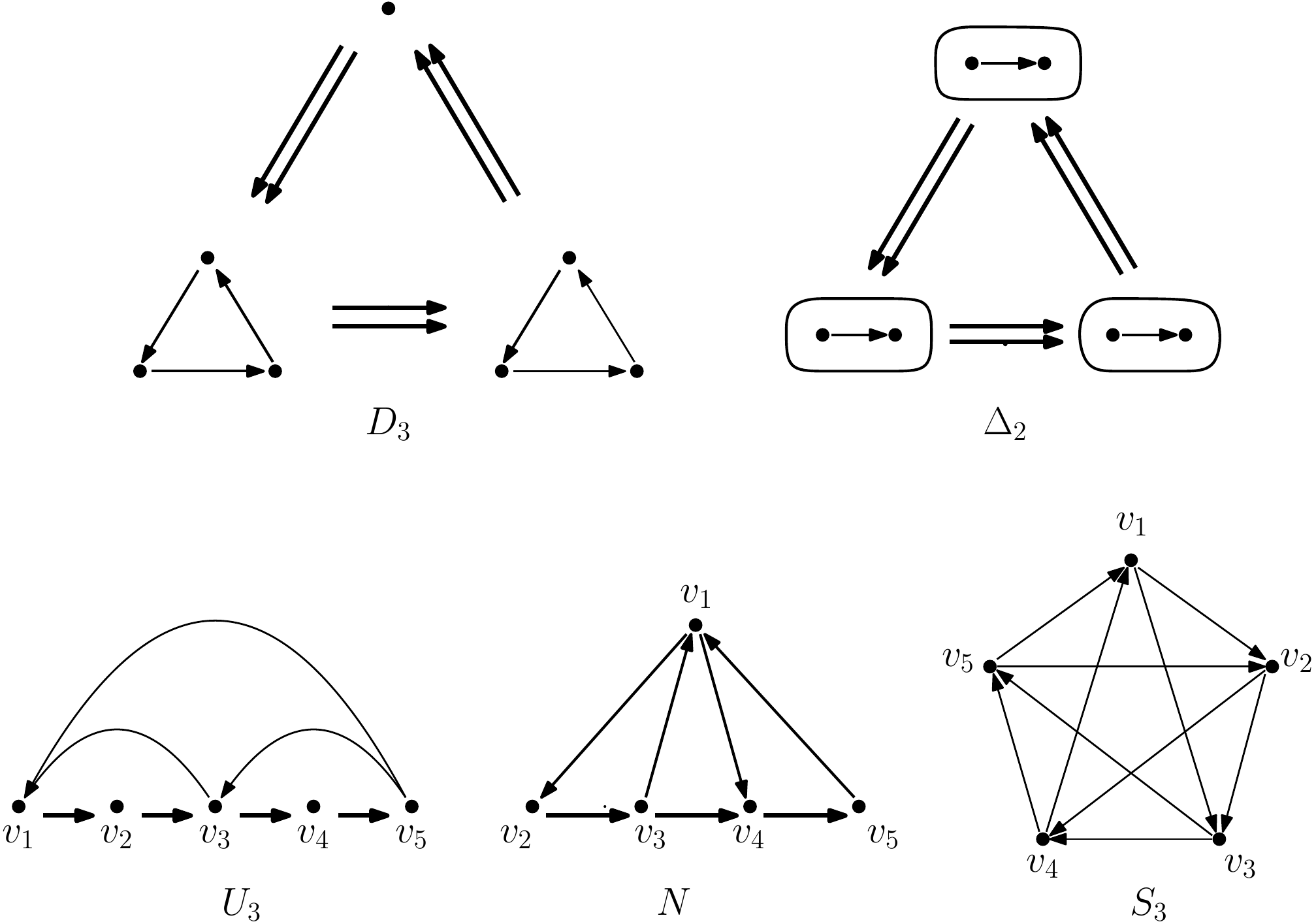}
\caption{Minimal non-heroes}
\end{figure}

In this section, we prove the following lemmas, 
which will be used to prove Theorem~\ref{THM:main2}.  

\reD*

\begin{LEM}\label{LEM:A}
$\mathcal{A}$ contains $U_3$ and $\D_2$, but does not contain $D_3,N$ or $S_3$.
\end{LEM}

\begin{LEM}\label{LEM:F}
$\mathcal{F}$ contains $U_3$ and $N$, but does not contain $D_3,S_3$ or $\D_2$.
\end{LEM}

We remark that Lemma~\ref{lem:nonhero-d3} is equivalent to that $D_3$ is the only minimal non-hero contained in $\mathcal{D}$.
Before proving the lemmas, we note that $U_3, N$ and $S_3$ are prime, that is, they do not contain homogeneous sets. We also note that every minimal non-hero is strongly connected.

\begin{proof}[Proof of Lemma~\ref{lem:nonhero-d3}]

Clearly, $D_3 \in \mathcal{D}$ by the definition of $\mathcal{D}$.

Suppose either $U_3$, $N$,  $S_3$ or $\D_2$ is contained in $\mathcal{D}$, say $X$. 
Let $k$ be the minimum integer such that $D_k$ contains $X$. 
Since $X$ is contained in $D_k$ but not in $D_{k-1}$, and $X$ is strongly connected, 
it follows that $X$ has a trisection $(A,B,C)$ with $|A|=1$, say $A=\{a\}$.
So, $X\setminus a$ is not strongly connected. 
Since there is no such vertex $a$ in $\D_2$, $X$ is either $U_3$, $N$ or $S_3$, which implies that $X$ is prime. 
So, $B$ and $C$ also contain only one vertex, which yields  a contradiction since $|V(X)|=5$.
This completes the proof.
\end{proof}

\begin{proof}[Proof of Lemma~\ref{LEM:A}]
Clearly, $U_3,\D_2 \in \mathcal{A}$ since $A_3$ contains $U_3$ and $A_5$ contains $\D_2$. 
It is also trivial that $N, S_3 \not \in \mathcal{A}$ by Proposition~\ref{Uprime} since $N$ and $S_3$ are prime but not isomorphic to $U_3$.

To show that $D_3 \not\in \mathcal{A}$, 
suppose $D_3 \in \mathcal{A}$. Since $D_3$ is strongly connected, it has a $\D$-partition $(\{v_1\},X_1,\{v_2\},X_2,\ldots,X_{n-1},\{v_n\})$ by Proposition~\ref{PROP:upartition}. 
Let $(\{x\},Y,Z)$ be a trisection of $D_3$ where $D_3|Y$ and $D_3|Z$ are cyclic triangles. 
Since $Y$ and $Z$ are maximal homogeneous sets of $D_3$, it follows that  $Y=X_i$ and $Z=X_j$ for some $i<j$ by Proposition~\ref{Uhomo}.
Thus, we obtain that $n\ge 3$, which is a contradiction 
since $7=|V(D_3)|\ge n+ |X_i|+|X_j| \ge 9$.
Therefore, $D_3$ does not belong to $\mathcal{A}$.
\end{proof}

\begin{proof}[Proof of Lemma~\ref{LEM:F}]
Observe that $U_3$ has a forest ordering $v_1,v_2,v_3,v_4,v_5$ with backward edges $\{v_1v_4,v_2v_5 \}$, 
and $N$ has a forest ordering $u_1,u_2,u_3,u_4,u_5$ with backward edges $\{u_1u_3,u_1u_5\}$. 
Hence, $U_3, N\in \mathcal{F}$.

Since $D_3$ has chromatic number three,  $D_3$ is not a forest tournament by Proposition~\ref{twocolor}.

To prove $S_3 \not \in \mathcal{F}$, suppose $S_3$ is a forest tournament, and $\sigma$ is its forest ordering. 
Let $V(S_3)= \{v_1,v_2,\ldots,v_5\}$ such that $v_i$ is adjacnet to $v_j$ if and only if $j-i \equiv \{1,2\} \mod 5$. 
Since $S_3$ is vertex-transitive, 
we may assume that $v_1$ is the first vertex in $\sigma$. The in-neighbors of $v_1$ are $v_4$ and $v_5$, so $v_4<_{\sigma}v_5$ since otherwise, $\{v_1,v_4,v_5\}$ induces a triangle in $B_{\sigma}(S_3)$.
If $v_2<_{\sigma}v_5$, then $B_{\sigma}(S_3)|\{v_1,v_2,v_4,v_5\}$ has thickness two, but it has only one component which yields a contradiction by Proposition~\ref{onecomponent}. Thus, $v_5<_{\sigma}v_2$, and we have $\sigma|\{v_1,v_2,v_4,v_5\}=v_1,v_4,v_5,v_2$.
Then, no matter where $v_3$ is, $B_{\sigma}(S_3)$ is connected and has thickness two, 
again a contradiction  by Proposition~\ref{onecomponent}. 
Hence, $S_3 \not \in \mathcal{F}$.

Suppose $\D_2$ is a forest tournament. Let $(\{v_1,v_2\},\{v_3,v_4\},\{v_5,v_6\})$ be a trisection of $\D_2$ where $v_i$ is adjacent to $v_{i+1}$ for $i=1,3,5$,
and $\sigma$ be a forest ordering of $\D_2$. 
We may assume that either $v_1$ or $v_2$ is the first vertex in $\sigma$. 

Let $v_1$ be the first vertex in $\sigma$. Then, $v_5 <_{\sigma} v_6$, since otherwise $\{v_1,v_5,v_6\}$ induces a triangle in $B_{\sigma}(\D_2)$.  
If $v_6<_{\sigma} v_3$, then $\{v_1,v_3,v_5,v_6\}$ induces a cycle in $B_{\sigma}(\D_2)$. So, $v_3 <_{\sigma} v_6$, and, by the same reason, $v_4 <_{\sigma} v_6$.

If $v_2 <_{\sigma} v_5$, then $B_{\sigma}(\D_2)|\{v_1,v_2,v_5,v_6\}$ is a cycle of length four.
If $v_5 <_{\sigma} v_2<_{\sigma} v_6$, then  the induced subgraph $B_{\sigma}(\D_2)|\{v_1,v_2,v_3, v_5,v_6\}$ is connected and has thickness two.
So, Corollary~\ref{nocycle} and Proposition~\ref{onecomponent} imply
 $v_6 <_{\sigma}v_2$, and so $v_3 <_{\sigma} v_4$ since otherwise, $\{v_2,v_3,v_4\}$ induces a triangle in $B_{\sigma}(\D_2)$. 
If $v_5<_{\sigma} v_4 <_{\sigma}v_6$, then $B_{\sigma}(\D_2)$ is connected and has thickness two, so Proposition~\ref{onecomponent} implies that $v_3<_{\sigma} v_4 <_{\sigma} v_5$.  
Now we have $\sigma=v_1,v_3,v_4,v_5,v_6,v_2$ and the backward edges $v_5v_1,v_6v_1,v_2v_3,v_2v_4$. 
However, in this case, there is no forest cut of $\D_2$ under $\sigma$. It is a contradiction.

Hence, $v_2$ is the first vertex in $\sigma$. Again, $v_5 <_{\sigma} v_6$. If $v_1 <_{\sigma} v_6$, then $\{v_1,v_2,v_6\}$ induces a triangle in $B_{\sigma}(\D_2)$. Hence, $v_6 <_{\sigma}v_1$. Then, no matter where $v_3$ is, $B_{\sigma}(\D_2)|\{v_1,v_2,v_3,v_5,v_6\}$ is connected and has thickness two, a contradiction by Proposition~\ref{onecomponent}. Therefore $\D_2 \not \in \mathcal{F}$.
\end{proof}

\section{Characterization of tournaments in $\mathcal{A}\cap \mathcal{F}$}\label{SEC:mainproof}

In this section, we characterize all tournaments in $\mathcal{A}\cap \mathcal{F}$. 
We simply write $\mathcal{AF}$ for $\mathcal{A}\cap \mathcal{F}$.

\begin{THM}\label{THM:main2-1}
Let $H$ be a tournament. Then, $H \in \mathcal{AF}$ if and only if it is isomorphic to one of the following.
\begin{itemize}
\item[1)]  $I$;
\item[2)] $H_1\Rightarrow H_2$ for $H_1,H_2 \in \mathcal{AF}$;
\item[3)] $\Delta(I,L_k,H')$ or $\Delta(I,H',L_k)$ for an integer $k$ and $H'\in \mathcal{AF}$;
\item[4)] $\D(L_{k_1},I,L_{k_2},L_{k_3}, I)$ or $\D(I,L_{k_1},L_{k_2},I, L_{k_3})$ for integers $k_1,k_2,k_3$;
\item[5)] $\D(I,H',L_{k_1},L_{k_2}, I)$ or $\D(I,L_{k_1},L_{k_2},H', I)$ for integers $k_1$, $k_2$ and $H' \in \mathcal{AF}$;
\item[6)] $\D(I,L_{k_1}, L_{k_2}, I,L_{k_3},L_{k_4},I)$ for integers $k_1,k_2,k_3,k_4$.
\end{itemize}
\end{THM}

Theorem~\ref{THM:main2-1} implies  Theorem~\ref{THM:main2} as follows.

\begin{proof}[Proof of Theorem~\ref{THM:main2}, assuming Theorem~\ref{THM:main2-1}]
Let $D\in \mathcal{D}'$ and $H$ be a $D$-hero. 
Lemma~\ref{LEM:A} together with Lemma~\ref{LEM:F} imply $D_3 \notin \mathcal{A} \cup \mathcal{F}$. So, since $D$ contains $D_3$, $D \notin A \cup F$.
So, by Theorem~\ref{THM_meet1} and Theorem~\ref{THM_meet2}, $H$ must belong to $\mathcal{AF}$.
Therefore, $H$ is a tournament in the list of Theorem~\ref{THM:main2-1}. 
This completes the proof.
\end{proof}

In order to prove Theorem~\ref{THM:main2-1}, we need the following lemmas. Observe that $\mathcal{AF}$ is hereditary since both $\mathcal{A}$ and $\mathcal{F}$ are hereditary.
We denote by $C$ a cyclic triangle.

\begin{LEM}\label{AF3}
Let $H=\Delta(G_1,G_2,G_3)$ for some tournaments $G_1,G_2,G_3$. If $H\in \mathcal{AF}$, then, either $G_1$, $G_2$ or $G_3$ is a one-vertex tournament, and one of the others is transitive.
\end{LEM}
\begin{proof}
If $|V(G_i)|\ge2$ for every $i=1,2,3$, then $H$ contains $\D_2$, which is a contradiction since $\D_2 \not \in \mathcal{F}$ by Lemma~\ref{LEM:F}. So,  either $|V(G_1)|,|V(G_2)|$ or $|V(G_3)|$ is equal to one. 
We may assume $|V(G_1)|=1$. 

If both of $G_2$ and $G_3$ contain a cyclic triangle, 
then $T$ contains $D_3$, a contradiction since $D_3 \not \in \mathcal{AF}$. So,  either $G_2$ or $G_3$ is a transitive tournament. 
This completes the proof.
\end{proof}

\begin{LEM}\label{AF5}
Let $H=\D(G_1,G_2,G_3,G_4,G_5)$ for tournaments $G_i$. If $H\in \mathcal{AF}$ then $(G_1,G_2,G_3,G_4,G_5)$ is either 
\begin{itemize}
\item  $(L_{k_1},I,L_{k_2},L_{k_3}, I)$, $(I,L_{k_1},L_{k_2},I, L_{k_3})$, 
\item $(I,H',L_{k_1},L_{k_2},I)$ or $(I,L_{k_1},L_{k_2},H',I)$ 
\end{itemize}
for some integers $k_1,k_2,k_3$ and some $H'\in \mathcal{AF}$.
\end{LEM}
\begin{proof}
First, we claim that $G_1,G_3$ and $G_5$ are transitive tournaments.
\\

\noindent (1) {\em $G_1,G_3,G_5$ does not contain a cyclic triangle.}
\\
\\
Let $H_1=\D(C,I,I,I,I)$, $H_2=\D(I,I,C,I,I)$ and $H_3=\D(I,I,I,I,C)$.
Since $\mathcal{AF}$ is hereditary, 
it is enough to prove that neither $H_1$, $H_2$ nor $H_3$ is contained in $\mathcal{A}$. 
Suppose $H_k \in \mathcal{A}$ for some $k\in \{1,2,3\}$.
Note that $V(H_k)$ contains a maximal homogeneous set $S$ with three vertices inducing a cyclic triangle, and there are two vertices  complete to $S$ and there are two vertices complete from $S$. 

 Since $H_k$ is strongly connected, there exists a $\D$-partition  of $V(H_k)$, $(\{v_1\},X_1,\{v_2\},X_2,\ldots,X_{n-1},\{v_{n}\})$, by Proposition~\ref{PROP:upartition}. 
By Proposition~\ref{Uhomo}, there exists $r$ such that $S=X_r$. Observe that $|X_i|= 1$ for every $i(\neq r)$ since $S$ is the only maximal homogeneous set of $H_k$. So, we obtain the following inequality:
$$
7=|V(H_k)|=n+\sum_{i=1}^{n-1} |X_i| \ge n+3+(n-2)=2n+1.
$$
So, $n\le 3$. If $n=2$ then $r=1$, and $|V(H_k)|=n+|X_1|=5<7$, a contradiction. 
Hence, $n=3$. 
If $X_1=S$ (resp. $X_2=S$), then there exists only one vertex complete to $S$ (resp. complete from $S$). This yields a contradiction since in $H_k$, there are two vertices complete to $S$ and two vertices complete from $S$. 
Therefore, $H_k \not \in \mathcal{A}$ for $k=1,2,3$. This proves (1).
\\

\noindent (2) {\em $\D(I,C,I,I,L_2), \D(L_2,I,I,C,I) \not \in \mathcal{F}$. }
\\
\\
The complement of a forest tournament is a forest tournament. 
So, it is enough to prove that  $\D(I,C,I,I,L_2) \not \in \mathcal{F}$ since $\D(L_2,I,I,C,I)$ is the complement of $\D(I,C,I,I,L_2)$, 

Let $K=\D(I,C,I,I,L_2)$ and $V(K)=\{v_1,v_2,\ldots,v_8\}$ such that for $1\le i<j\le 8$, $v_i$ is adjacent from $v_j$ if and only if $(i,j)=(1,5)$, $(1,7)$, $(1,8)$, $(2,4)$, $(5,7)$, $(5,8)$ and $(7,8)$. Figure~\ref{pic1} describes all backward edges of $K$ under $\sigma=v_1,v_2,\ldots,v_8$.
\begin{figure} [ht!]
\centering
\includegraphics[scale=0.6]{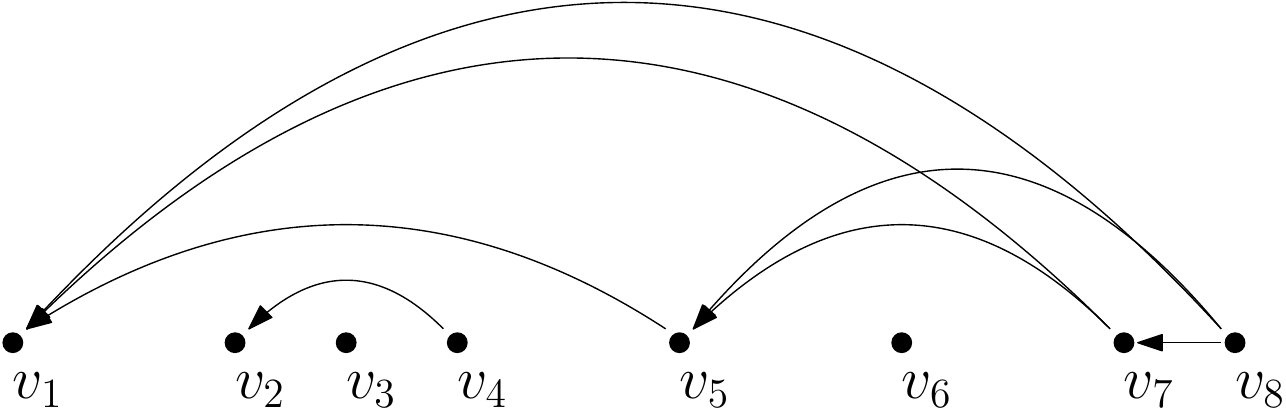}
\caption{All backward edges of $K$ under $\sigma$}
\label{pic1}
\end{figure}

Suppose $K$ is a forest tournament with forest ordering $\sigma$. 
Since $\{v_2,v_3,v_4\}$ induces a cyclic triangle, there exists a backward edge $v_b v_a$ with $a,b \in \{2,3,4\}$ under $\sigma$. 
Since $v_1$ is complete to $\{v_2,v_3,v_4\}$, if $v_b <_{\sigma} v_1$, then $\{v_1,v_a,v_b\}$ induces a cyclic triangle in $B_{\sigma}(K)$. So, $v_1 <_{\sigma} v_b$.
Similarly, for $i=5,6,7,8$, since $v_i$ is complete from $\{v_2,v_3,v_4\}$, $v_a <_{\sigma} v_i$.   
Suppose $v_i <_{\sigma} v_1$ for some $i=5,6,7,8$. Then, $v_a <_{\sigma} v_i <_{\sigma} v_1 <_{\sigma}v_b$, and $K|\{v_a,v_i,v_1,v_b\}$ has thickness two under $\sigma| \{v_a,v_i,v_1,v_b\}$, which yields a contradiction by Proposition~\ref{onecomponent}. Hence, $v_1 <_{\sigma} v_i$ for $i=5,6,7,8$. 

Since $B_{\sigma}(K)$ is a forest and $v_1 <_{\sigma} v_i$ for $i=5,7,8$, 
it follows that $v_8 <_{\sigma} v_7 <_{\sigma} v_5$.
Let us look at $v_6$. 
If $v_7 <_{\sigma} v_6$, then $\{v_1,v_8,v_7,v_6\}$ induces a cycle of length four in $B_{\sigma}(K)$.
So, $v_1 <_{\sigma} v_6 <_{\sigma}v_7$. However, in this case, $K|\{v_1,v_6,v_8,v_7,v_5\}$ has thickness two under $\sigma|\{v_1,v_6,v_8,v_7,v_5\}$, which is a contradiction. 
Therefore, $K$ is not a forest tournament. This prove (2).
\\

By (1), $G_1$, $G_3$ and $G_5$ are transitive tournaments. 

\noindent {\bf Case 1: $|V(G_1)|\ge 2$.} 
Then $G_2$ and $G_5$ are one-vertex tournaments, since otherwise, $V(G_1) \cup V(G_2) \cup V(G_5)$ induces a subtournament of $H$ containing $\D_2$ which is not a forest tournament by Lemma~\ref{LEM:F}.
By (2), $G_4$ does not contain a cyclic triangle. This implies that $(G_1,G_2,G_3,G_4,G_5)=(L_{k_1}, I, L_{k_2}, L_{k_3},I)$ for some positive integers $k_1,k_2$ and $k_3$. 

\noindent {\bf Case 2: $|V(G_5)|\ge 2$.} 
Similar to Case 1, we have $(G_1,G_2,G_3,G_4,G_5)=(I,L_{k_1}, L_{k_2},I, L_{k_3})$ for some positive integers $k_1,k_2$ and $k_3$. 

\noindent {\bf Case 3: $|V(G_1)|=|V(G_1)|=1$.} 
If both $G_2$ and $G_4$ contain a cyclic triangle, then $H$ contains $A_3$, which is not a forest tournament since $\chi(A_3)=3$. 
Thus, either $G_2$ or $G_4$ is a transitive tournament. 
Finally, since $\mathcal{AF}$ is hereditary, it follows that $G_2, G_4 \in \mathcal{AF}$, so
$(G_1,G_2,G_3,G_4,G_5)= (I,H',L_{k_1},L_{k_2},I)$ or $(I,L_{k_1},L_{k_2},H',I)$ for positive integers $k_1,k_2$ and  $H'\in \mathcal{AF}$. This completes the proof.
\end{proof}

\begin{LEM}\label{AF7}
Let $H=\D(G_1,G_2,G_3,G_4,G_5,G_6,G_7)$ for tournaments $G_i$s. If $H\in \mathcal{AF}$ then $G_i$ is transitive for $i=1,\ldots,7$ and $|V(G_j)|=1$ for $j=1,4,7$.
\end{LEM}

\begin{proof}
For some $j=1,4,7$, if $|V(G_j)|\ge 2$, then $H$ contains $\D_2$ which is not a forest tournament by Lemma~\ref{LEM:F}. So, $G_j$ is a one-vertex tournament for $j=1,4,7$.

If either $G_2, G_3, G_5$ or $G_6$ contains a cyclic triangle, then $T$ contains either $\D(I,I,C,I,I)$, $\D(I,C,I,I,L_2)$ or $\D(L_2,I,I,C,I)$, which does not belong to $\mathcal{AF}$ by Lemma~\ref{AF5}. So, $G_i$ is a transitive tournament for $i=2,3,5,6$.
This completes the proof.
\end{proof}

If $T_1$ and $T_2$ are tournaments with at least two vertices, 
then for a vertex $v \in V (T_1)$, we say \emph{a tournament $T$ is obtained from $T_1$ by substituting $T_2$
for $v$} if $V(T) = V(T_1) \cup V (T_2) \setminus \{v\}$  and $xy \in E(T)$ if and only if one of
the following holds.
\begin{itemize}
\item  $xy \in E(T_1 \setminus v)$ or $xy \in E(T_2)$,
\item $x \in V (T_1)$, $y\in V (T_2)$ and $xv \in V (T_1)$, 
\item $x \in V (T_2)$, $y \in V (T_1)$, and $vy \in V (H_1)$.
\end{itemize}
We remark that every non-prime tournament can be obtained from a prime tournament by substitutions.
Now we are ready to prove Theorem~\ref{THM:main2-1}.

\begin{proof}[Proof of Theorem~\ref{THM:main2-1}]
First, we prove the `only if' part.
Let $H \in \mathcal{AF}$. 
\\
\\
(1) \emph{If $H$ is prime, then $H$ is isomorphic to $I,L_2, U_2(=C), U_3$ or $U_4$.}
\\
\\
If $|V(H)|\le 2$, then $H$ is isomorphic to either $I$ or $L_2$, so we are done.

Assume that $|V(H)|\ge 3$.
Since $H$ is prime and belongs to $\mathcal{A}$, 
Proposition~\ref{Uprime} implies that $H$ is isomorphic to $U_n$ for some $n\ge 2$.
If $n\ge 5$, then $U_n$ contains $\D_2$, so $U_n \not\in \mathcal{F}$ by Lemma~\ref{LEM:F}.
Therefore, $H$ is isomorphic to either $U_2 (=C), U_3$ or $U_4$.
\\

Since $I \in \mathcal{AF}$, we can obtain $L_2$ from 2), $U_2$ from 3), $U_3$ from 4) and $U_4$ from 6).

Let us consider the case that $H$ is not prime, that is, 
$H$ can be obtained from some prime tournament $G_0$ with $|V(G_0)|>1$ by substituting $G_1,G_2,\ldots,G_n$ for vertices $v_1,v_2,\ldots,v_n$ of $G_0$.
For each $i$, since $G_i$ is a subtournament of $H$, it also belongs to $\mathcal{AF}$. In particular, $G_0 \in \mathcal{AF}$.
So, by (1), $G_0$ is isomorphic to either $L_2, U_2, U_3$ or $U_4$. 
Then, Lemma~\ref{AF3}, Lemma~\ref{AF5} and Lemma~\ref{AF7} imply that $H$ is isomorphic to one of the tournaments in the list of Theorem~\ref{THM:main2-1}.

Now we prove the `if' part, that is,  every tournament $H$ in the list of Theorem~\ref{THM:main2-1} belongs to $\mathcal{AF}$. 
Since the complement of a tournament in $\mathcal{AF}$ belongs to $\mathcal{AF}$, we need to consider the following six cases.

\noindent{\bf Case 1.} It is trivial when $H=I$. 
\\
\\
\noindent{\bf Case 2.} $H=H_1\Rightarrow H_2$ for some $H_1,H_2 \in \mathcal{AF}$.
\\
\\
To show $H\in \mathcal{A}$, 
choose $k_i$ such that $A_{k_i}$ contains $H_i$ for $i=1,2$. Let $K=\max\{k_1,k_2\}$. Then, $A_{K+1}$ contains $H$ since $A_{K+1}$ contains two copies $A_K^{(1)}$ and $A_K^{(2)}$ of $A_K$ where $V(A_K^{(1)})$ is complete to $V(A_K^{(2)})$.

To show that $H\in \mathcal{F}$, let $\sigma_i$ be a forest ordering of $H_i$ for $i=1,2$. Let $\sigma=\sigma_1,\sigma_2$, that is, for $u,v \in V(H)$, $v<_{\sigma} u$ if either 
\begin{itemize}
\item $v\in V(H_1)$ and $u\in V(H_2)$,
\item  $v,u \in V(H_1)$ and $v<_{\sigma_1} u$, or 
\item $v,u \in V(H_2)$ and $v<_{\sigma_2} u$. 
\end{itemize}
Then, $\sigma$ is a forest ordering of $H$, and so $H$ is a forest tournament. Therefore $H\in \mathcal{AF}$.
\\
\\
\noindent{\bf Case 3.} $H=\D(I,L_k,H')$ 
for an integer $k$ and $H' \in \mathcal{AF}$.
\\
\\
To show $H \in \mathcal{A}$, let $M$ be a positive integer such that  $A_{M-1}$ contains $H'$ and $M>k$. Then, clearly, $A_{M}$ contains $ \D(I,L_k,H')$.

To show $H\in \mathcal{F}$, let $v_1,v_2,\ldots, v_n$ be a forest ordering of $H'$, and $V(L_k)=\{u_i\mid 1\le i\le k\}$ with $u_i \to u_j$ for $1\le i <j \le k$. 
Then, the ordering $u_1,u_2,\ldots,u_k,v_1,v_2,\ldots,v_n,w$ is a forest ordering of $\Delta(I,L_k,H)$, so, $H$ is a forest tournament. Therefore, $H\in \mathcal{AF}$.
\\
\\
\noindent{\bf Case 4.} $H=\D(L_{k_1},I,L_{k_2},L_{k_3}, I)$ 
for integers $k_1,k_2,k_3$;
\\
\\
To prove $H \in \mathcal{A}$, let $K=k_1+k_2+k_3$. Then, $U_K$ contains $\D(L_{k_1},I,L_{k_2},L_{k_3}, I)$ 
and so, $\D(L_{k_1},I,L_{k_2},L_{k_3},I)$ belongs to $\mathcal{A}$.

Let $\sigma=a,b,y_1,y_2,\ldots,y_{k_2},x_1,x_2,\ldots,x_{k_1},z_1,z_2,\ldots,z_{k_3}$ be a vertex ordering of $\D(L_{k_1},I,L_{k_2},L_{k_3},I)$ with backward edges $\{x_{i}a \mid 1\le i \le k_1\} \cup \{z_jb\mid  1\le j \le k_3\}$. Then, $\sigma$ is a forest ordering of $H$. So $H$ is  a forest tournament.
\\
\\
\noindent{\bf Case 5.} $H=\D(I,H',L_{k_1},L_{k_2}, I)$ 
for integers $k_1$, $k_2$ and $H' \in \mathcal{AF}$.
\\
\\
Let $M>k_1+k_2$ be a positive integer such that $A_{M-1}$ contains $H'$. Then, $A_M$ contains $H$, so $H \in \mathcal{A}$.

To prove $H \in \mathcal{F}$,  let $\sigma'=v_1,v_2,\ldots,v_n$ be a forest ordering of $H'$. Let $\sigma=a,v_1,v_2,\ldots,v_n, b,x_1,x_2,\ldots,x_{k_1},y_1,\ldots,y_{k_2}$ be a vertex ordering of $H$ with backward edge set  the union of the set of backward edges of  $H'$ under $\sigma'$, $\{ab\}$, $\{x_ia\mid 1\le i \le k_1\}$ and $\{y_ib \mid 1\le i \le k_2\}$. Then, $\sigma$ is a forest ordering. So, $H \in \mathcal{F}$.
\\
\\
\noindent{\bf Case 6.} $H=\D(I,L_{k_1}, L_{k_2}, I,L_{k_3},L_{k_4},I)$ for integers $k_1,k_2,k_3,k_4$.
\\
\\
Let 
$$
\sigma=x_1,\ldots,x_{k_1},c,z_1,\ldots,z_{k_3},b,y_1,\ldots,y_{k_2},a,w_1,\ldots,w_{k_4}
$$
be a vertex ordering of $\D(I,L_{k_1}, L_{k_2}, I,L_{k_3},L_{k_4},I)$ where the set of  backward edges is $\{(a,x_i) \mid i\in [k_1]\} \cup \{(w_i,c)\mid i\in [k_4]\} \cup \{(b,c)\} \cup \{(b,z_i)\mid i\in [k_3]\} \cup \{(y_i,b)\mid i\in [k_2]\} \cup \{(a,b)\}$.
Then, $\sigma$ is a forest ordering, so $\D(I,L_{k_1}, L_{k_2}, I,L_{k_3},L_{k_4},I)$ is a forest tournament. 

To prove $\D(I,L_{k_1}, L_{k_2}, I,L_{k_3},L_{k_4},I)\in \mathcal{A}$, let $K=k_1+k_2+k_3+k_4+1$. Then, $U_4(I,L_{k_1}, L_{k_2}, I,L_{k_3},L_{k_4},I)$ is contained in $A_K$.

This completes the proof.
\end{proof}

\section{Constructions of heroes}\label{SEC:growing}

In this section, we prove Theorem~\ref{THM:growing2} and Theorem~\ref{THM:U3}.

\subsection{Proof of Theorem~\ref{THM:growing2}}

We start with the following observations.

\begin{OBS}\label{OBS:hero_property}
Let $D\in \mathcal{D}'$, and $\mathcal{H}$ be the set of all $D$-heroes.
\begin{itemize}
\item[(1)] $\mathcal{H}$ is hereditary since  for a tournament $T$ and its subtournament $T'$, every $T'$-free tournament is $T$-free. 
\item[(2)] $\mathcal{H}$ is closed under taking complement since every tournament has the same chromatic number with its complement.
\end{itemize}
\end{OBS}

We need a lemma from~\cite{hero} in order to 
prove Theorem~\ref{THM:growing2}.
We first give the following definitions.

For tournaments $G, H$ and an integer $a$, an {\em $(a,G,H)$-jewel} is a tournament $T$ with $|V(T)|=a$ such that every partition $(A,B)$ of $V(T)$, either $T|A$ contains $G$ or $T|B$ contains $H$.
We say a tournament $T$ \emph{contains an $(a,G,H)$-jewel chain of length $n$} if there exist vertex disjoint subtournaments  $J_1,J_2,\ldots,J_n$ of $T$, which are $(a,G,H)$-jewels, such that $V(J_i)$ is complete to $V(J_j)$ for $1\le i<j\le n$.  
	
\begin{LEM}[Berger et al. \cite{hero}]\label{lemma2}
Let $H$, $K$ be tournaments and $a\ge 1$ an integer. If either $H$ or $K$ is transitive, then there is a map $f_{H,K}:\mathbb{Z}^+\times \mathbb{Z}^+ \to \mathbb{Z}^+$ satisfying the following property.
For every $\Delta(I,H,K)$-free tournament $G$, if
\begin{itemize}
\item $c_1$ is an integer such that every $H$-free subtournament of $G$  and $K$-free subtournament of $G$ has chromatic number at most $c_1$, and
\item $c_2$ is an integer such that every subtournament of $G$ containing no $(a,H,K)$-jewel-chain of length four has chromatic  number at most $c_2$,
\end{itemize}
then $G$ has chromatic number at most $f_{H,K}(c_1,c_2)$.
\end{LEM}

We also need the following result of Stearns~\cite{transitive}.

\begin{THM}[Stearns~\cite{transitive}]\label{THM:Ramsey}
For every integer $k\ge 1$, every tournament with at least $2^{k-1}$ vertices contains $L_k$.
\end{THM}

\begin{proof}[Proof of Theorem~\ref{THM:growing2}]
Let $D \in \mathcal{D}'$ and $H$ a tournament admitting a trisection.

Suppose $H$ is a $D$-hero. Then, $H$ belongs to $\mathcal{AF}$, and by Lemma~\ref{AF3}, $H$ is isomorphic to $\D(I,H', L_k)$ or $\D(I,L_k,H')$ for some positive integer $k$ and $H' \in \mathcal{AF}$. Since $H'$ is a subtournament of $H$, it is a $D$-hero by Observation~\ref{OBS:hero_property} (1).
This proves the `only if' part.

For the `if' part, it is enough to show that $H=\D(I,H',L_k)$ is a $D$-hero by Observation~\ref{OBS:hero_property} (2). Let $c$ be an integer such that every $H'$-free tournament has chromatic number at most $c$.
We show that there exists $d$ such that every $H$-free tournament $T$ has chromatic number at most $d$. 
 Let $a=2^k|V(H')|$.
\\
\\
(1) For every tournament  $T'$, if $T'$ contains no $(a,H',L_k)$-jewels, then $\chi(T')$ is less than $a+c$.
\\
\\
If $T'$ is $H'$-free, then $\chi(T') \le c$. So, we may assume that $T'$ contains $H'$.
Let $H_1,H_2,\ldots,H_m$ be vertex disjoint subtournaments of $T'$ isomorphic to $H'$ with $m$ maximum. 
Let $J=\bigcup_{i=1}^{\min\{m,2^k\}} V(H_i)$.

If $m\ge 2^k$, then  $T'|J$ is an $(H',L_k)$-jewel. (For every partition $(X,Y)$ of  $J$,  if $T'|X$ is $H'$-free, then $Y$ meets $V(H_i)$ for every $i=1,2,\ldots,2^k$, which implies $|Y| \ge 2^k$. So, $T'|Y$ contains $L_k$ by Theorem~\ref{THM:Ramsey}.)
Thus, $m<2^k$. 
Note that $T'\setminus J$ is $H'$-free by the maximality of $m$, so it has chromatic number at most $c$. 
Therefore, the chromatic number of $T'$ is at most 
$$
\chi(T'|J)+\chi(T'\setminus J) \le |J|+c < a+c.
$$
This proves (1).
\\

By (1), we may assume that $T$ contains $(a,H',L_k)$-jewels.
Let $\mathcal{J}$ be the set of all $(a,H',L_k)$-jewels, $\mathcal{J}_1=\{J'\Rightarrow J''|J',J'' \in \mathcal{J}\}$ 
and $\mathcal{J}_2=\{J_1' \Rightarrow J_1''|J_1',J_1'' \in \mathcal{J}_1\}$. 
Since every $\mathcal{J}$-free subtournament of $T$ has chromatic number at most $a+c$ by (1),  
every $\mathcal{J}_1$-free subtournament of $T$ has chromatic number at most  some constant $c_1$ by Lemma~\ref{lemma1} with $\mathcal{H}_1=\mathcal{H}_2=\mathcal{J}$. 
By applying Lemma~\ref{lemma1} again, 
there exists $c_2$ such that every $\mathcal{J}_2$-free subtournament of $T$ 
has chromatic number at most $c_2$.
Since every tournament not containing $(a,H',L_k)$-jewel-chain of length four is $\mathcal{J}_2$-free, it has chromatic number at most $c_2$.
Hence, by Lemma~\ref{lemma2}, there exists $d$ such that every $H$-free tournament has chromatic number at most $d$. This completes the proof.
\end{proof}

\subsection{Proof of Theorem~\ref{THM:U3}}\label{SEC:proofU3}

To prove Theorem~\ref{THM:U3}, we need the following result of Liu~\cite{gaku}.
(Recall that $S_n$ is a tournament defined at the beginning of Section~\ref{SEC:nonhero}.) 
 
\begin{THM}\label{gaku}(Liu~\cite{gaku})
Let $T$ be a prime tournament. Then, $T$ is $U_3$-free if and only if $T$ is isomorphic to $S_n$ for some $n\ge1$ or $V(T)$ can be partitioned into sets $X_1,X_2,X_3$ such that $X_1\cup X_2$, $X_2\cup X_3$ and $X_3\cup X_1$ are transitive.
\end{THM}

\begin{proof}[Proof of Theorem~\ref{THM:U3}]
We prove that for a tournament $T$, if $T$ is $\{D_n,U_3\}$-free for some $n\ge 2$, then $T$ is $3^{n-2}$-colorable. 
This implies Theorem~\ref{THM:U3} since for every $D\in \mathcal{D}'$, there exists $D_n$ containing $D$, and every $\{D,U_3\}$-free tournament is $\{D_n,U_3\}$-free.
We use induction on $|V(T)|$.

The base case is that either $n=2$ or $T$ is prime.
If $n=2$,  then $\chi(T) =1$ since $T$ is $D_2$-free. 
If $T$ is prime and $n>2$, then Theorem~\ref{gaku} implies that $G$ is two colorable, so we are done.

Suppose $T$ is not prime, and assume the statement is true for every graph with less than $|V(T)|$ vertices.

Let $T$ be $\{D_n, U_3\}$-free for some $n>2$. 
Since $T$ is not prime, $T$ is obtained from some prime tournament $G_0$ by substitutions, and Theorem~\ref{gaku} implies that  either
\begin{itemize}
\item $G_0$ is isomorphic to $S_m$ for some $m\ge 2$, or 
\item $V(G_0)$ can be partitioned into sets $X_1,X_2,X
_3$ such that  $X_1\cup X_2$, $X_2\cup X_3$ and $X_3\cup X_1$ are transitive. 
\end{itemize}

For the first case, let $V(G_0)=\{v_1,v_2,\ldots,v_{2m-1}\}$ with $v_iv_j \in E(G_0)$ for every $i,j$ with $j-i \in \{1,2,\ldots,m-1\} \mod{2m-1}$, and let $T$ be obtained from $G_0$ by substituting $G_i$ for $v_i$ for $i=1,2,\ldots,2m-1$.
For every edge $v_iv_j$ of $G_0$, there exists a vertex $v_k$ such that $v_j \to v_k$ and $v_k\to v_i$. 
Hence, if both $G_i$ and $G_j$ contain $D_{n-1}$, then for every $v\in V(G_k)$, $V(G_i) \cup V(G_j) \cup \{v\}$ induces a subtournament of $T$ containing $D_n$, which yields a contradiction. 
Therefore, all tournaments $G_1,G_2,\ldots,G_{2m-1}$ but one are $D_{n-1}$-free.
Without loss of generality, let $G_1,G_2,\ldots,G_{2m-2}$ be $D_{n-1}$-free. 

By the induction hypothesis,  
there exist a $3^{n-3}$-coloring $\phi_i:V(G_i) \to \{1,2,3\}^{n-3}$ of $G_i$ for $i=1,2,\ldots,2m-2$
and  a $3^{n-2}$-coloring $\phi_{2m-1}:V(G_{2m-1}) \to \{1,2,3\}^{n-2}$ of $G_{2m-1}$. 

We define a map $\phi: V(T)\to \{1,2,3\}^{n-2}$ as follows:
\begin{itemize}
\item for $v\in V(G_{2m-1})$, $\phi(v)=\phi_{2m-1}(v)$;
\item for $i=1,2,\ldots,m-1$ and $v\in V(G_i)$, $\phi(v)=\{1\} \times \phi_i(v)$, and
\item for $i=m,m+1,\ldots,2m-2$ and $v\in V(G_i)$, 
$\phi(v)=\{2\} \times \phi_i(v)$.
\end{itemize}
We claim that $\phi$ is a $3^{n-2}$-coloring of $T$.
Suppose there exist three vertices $u\in V(G_i)$, $v\in V(G_j)$ and $w\in V(G_k)$ inducing a monochromatic cyclic triangle in $T$. Clearly, $i$, $j$ and $k$ are all distinct. Indeed, $\{i,j,k\}$ intersects with only one of $\{1,2,\ldots,m-1\}$ and $\{m,m+1,\ldots, 2m-2\}$ by the definition of $\phi$.
However, in either case, $\{u,v,w\}$ is transitive, a contradiction.
Therefore, $\phi$ is a $3^{n-2}$-coloring of $T$.

For the second case, let $\{G_v \mid v \in V(G_0)\}$ be tournaments such that 
$T$ is  obtained from $G_0$ by substituting $G_v$ for $v$ for every $v\in G_0$. 
	Clearly, $G_v$ is $\{D_n, U_3\}$-free for $v \in V(G_0)$. 
	For each $v\in V(G_0)$, 
	let $\phi_v: V(G_v) \to \{1,2,3\}^{n-3}$ be a $3^{n-3}$-coloring of $G_v$ 
	if $G_v$ is $D_{n-1}$-free, 
	and $\phi_v:V(G_v) \to \{1,2,3\}^{n-2}$ be a $3^{n-2}$-coloring of $G_v$ 
	if $G_v$ contains $D_{n-1}$. 
	We define a map $\phi:V(G)\to \{1,2,3\}^{n-2}$ as follows. 
\begin{itemize}
\item for $v\in V(G_0)$ and $u\in G_v$, if $G_v$ contains $D_{n-1}$, then let $\phi(u)=\phi_v(u)$;
\item for $v\in v(G_0)$ and $u\in G_v$, if $G_v$ is $D_{n-1}$-free and $v \in X_i$, 
 then let $\phi(u)=\{i\} \times \phi_v(u)$.
\end{itemize}

We claim that $\phi$ is a $3^{n-2}$-coloring of $T$.
Suppose there exist three vertices $u_1\in V(G_{v_1})$, $u_2\in V(G_{v_2})$ and $u_3\in V(G_{v_3})$ inducing a monochromatic cyclic triangle in $T$. Clearly, $v_1$, $v_2$ and $v_3$ are all distinct vertices of $G_0$. 
If two of $G_{v_1}$, $G_{v_2}$ and $G_{v_3}$ contain $D_{n-1}$, then $V(G_{v_1})\cup V(G_{v_2}) \cup V(G_{v_3})$ induces a subtournament of $T$ containing $D_n$, which yields a contradiction. So, without loss of generality, let $G_{v_1}$ and $G_{v_2}$ be $D_{n-1}$-free. 
Then, for $u_1$ and $u_2$ to have the same color in $\phi$, $v_1$ and $v_2$ belong to the same $X_{\ell}$ for some $\ell =1,2,3$. 
Then, by the condition that $X_1 \cup X_2$, $X_2 \cup X_3$ and $X_3\cup X_1$ are transitive in $G_0$,  $\{v_1,v_2,v_3\}$ is transitive in $G_0$ and so $\{u_1,u_2,u_3\}$ is transitive in $T$, a contradiction. Therefore, $\phi$ is a $3^{n-2}$-coloring of $G$. This completes the proof.
\end{proof}

\bibliographystyle{plain}

\cleardoublepage
\ifdefined\phantomsection
  \phantomsection  
\else
\fi
\addcontentsline{toc}{chapter}{Bibliography}

\bibliography{ref}

\begin{thebibliography}{10}

\bibitem{AlKo2016}
Noga Alon, Alexandr Kostochka, Benjamin Reiniger, Douglas~B. West, and Xuding
  Zhu.
\newblock Coloring, sparseness and girth.
\newblock {\em Israel J. Math.}, 214(1):315--331, 2016.

\bibitem{hero}
Eli Berger, Krzysztof Choromanski, Maria Chudnovsky, Jacob Fox, Martin Loebl,
  Alex Scott, Paul Seymour, and St{\'e}phan Thomass{\'e}.
\newblock Tournaments and colouring.
\newblock {\em J. Combin. Theory Ser. B}, 103(1):1--20, 2013.

\bibitem{perfect}
Maria Chudnovsky, Neil Robertson, Paul Seymour, and Robin Thomas.
\newblock The strong perfect graph theorem.
\newblock {\em Ann. of Math. (2)}, 164(1):51--229, 2006.

\bibitem{hole3}
Maria Chudnovsky, Alex Scott, and Paul Seymour.
\newblock Induced {S}ubgraphs of {G}raphs with {L}arge {C}hromatic {N}umber.
  {III}. {L}ong {H}oles.
\newblock {\em Combinatorica}, 37(6):1057--1072, 2017.

\bibitem{hole12}
Maria Chudnovsky, Alex Scott, and Paul Seymour.
\newblock Induced subgraphs of graphs with large chromatic number. xii. distant
  stars, 2017.

\bibitem{hole8}
Maria Chudnovsky, Alex Scott, Paul Seymour, and Sophie Spirkl.
\newblock Induced subgraphs of graphs with large chromatic number. viii. long
  odd holes, 2017.

\bibitem{graphhero}
Maria Chudnovsky and Paul Seymour.
\newblock Extending the {G}y\'arf\'as-{S}umner conjecture.
\newblock {\em J. Combin. Theory Ser. B}, 105:11--16, 2014.

\bibitem{erdos}
P.~Erd\H{o}s.
\newblock Graph theory and probability.
\newblock {\em Canad. J. Math.}, 11:34--38, 1959.

\bibitem{gyarfas}
A.~Gy{\'a}rf{\'a}s.
\newblock On {R}amsey covering-numbers.
\newblock In {\em Infinite and finite sets ({C}olloq., {K}eszthely, 1973;
  dedicated to {P}. {E}rd{\H o}s on his 60th birthday), {V}ol. {II}}, pages
  801--816. Colloq. Math. Soc. Jan\'os Bolyai, Vol. 10. North-Holland,
  Amsterdam, 1975.

\bibitem{gyarfas3}
A.~Gy\'arf\'as.
\newblock Problems from the world surrounding perfect graphs.
\newblock {\em Tanulm\'anyok---MTA Sz\'amit\'astech. Automat. Kutat\'o Int.
  Budapest}, (177):53, 1985.

\bibitem{gyarfas5}
A.~Gy\'arf\'as, E.~Szemer\'edi, and Zs. Tuza.
\newblock Induced subtrees in graphs of large chromatic number.
\newblock {\em Discrete Math.}, 30(3):235--244, 1980.

\bibitem{gyarfas2}
H.~A. Kierstead and S.~G. Penrice.
\newblock Radius two trees specify {$\chi$}-bounded classes.
\newblock {\em J. Graph Theory}, 18(2):119--129, 1994.

\bibitem{gyarfas7}
H.~A. Kierstead and Yingxian Zhu.
\newblock Radius three trees in graphs with large chromatic number.
\newblock {\em SIAM J. Discrete Math.}, 17(4):571--581, 2004.

\bibitem{Kv1989}
Igor K\v{r}\'{i}\v{z}.
\newblock A hypergraph-free construction of highly chromatic graphs without
  short cycles.
\newblock {\em Combinatorica}, 9(2):227--229, 1989.

\bibitem{gaku}
Gaku Liu.
\newblock Structure theorem for {$U_5$}-free tournaments.
\newblock {\em J. Graph Theory}, 78(1):28--42, 2015.

\bibitem{Lo1968}
L.~Lov\'asz.
\newblock On chromatic number of finite set-systems.
\newblock {\em Acta Math. Acad. Sci. Hungar.}, 19:59--67, 1968.

\bibitem{My1955}
J.~Mycielski.
\newblock Sur le coloriage des graphs.
\newblock {\em Colloq. Math.}, 3:161--162, 1955.

\bibitem{tournamentcoloring}
V.~Neumann~Lara.
\newblock The dichromatic number of a digraph.
\newblock {\em J. Combin. Theory Ser. B}, 33(3):265--270, 1982.

\bibitem{NeRo1989}
Jaroslav Ne\v{s}et\v{r}il and Vojt\v{e}ch R\"odl.
\newblock Chromatically optimal rigid graphs.
\newblock {\em J. Combin. Theory Ser. B}, 46(2):133--141, 1989.

\bibitem{gyarfas6}
A.~D. Scott.
\newblock Induced trees in graphs of large chromatic number.
\newblock {\em J. Graph Theory}, 24(4):297--311, 1997.

\bibitem{hole1}
Alex Scott and Paul Seymour.
\newblock Induced subgraphs of graphs with large chromatic number. {I}. {O}dd
  holes.
\newblock {\em J. Combin. Theory Ser. B}, 121:68--84, 2016.

\bibitem{transitive}
Richard Stearns.
\newblock The voting problem.
\newblock {\em Amer. Math. Monthly}, 66:761--763, 1959.

\bibitem{sumner}
D.~P. Sumner.
\newblock Subtrees of a graph and the chromatic number.
\newblock In {\em The theory and applications of graphs ({K}alamazoo, {M}ich.,
  1980)}, pages 557--576. Wiley, New York, 1981.

\end{thebibliography}

\end{document}